\documentclass[reqno,11pt,a4paper]{article}
\usepackage{amsmath,amssymb,amsthm,color,graphicx}
\usepackage{algorithm}
\usepackage{algorithmic}

\newtheorem{theorem}{Theorem}[section]
\newtheorem{proposition}[theorem]{Proposition}
\newtheorem{definition}[theorem]{Definition}
\newtheorem{lemma}[theorem]{Lemma}

\newtheorem{remark}[theorem]{Remark}

\newtheorem{example}[theorem]{Example}

\numberwithin{equation}{section}
\numberwithin{theorem}{section}

\renewcommand{\phi}{\varphi}

\newcommand{\eps}{\varepsilon}

\newcommand{\nada}[1]{}

\newcommand{\R}{\mathbb R}
\renewcommand{\H}{{\mathcal H}}

\newcommand{\bd}{\partial}
\newcommand{\beq}{\begin{equation}}
\newcommand{\eeq}{\end{equation}}

\definecolor{darkgreen}{rgb}{0,0.55,0}

\newcommand{\res}{\mathop{\hbox{\vrule height 7pt width .5pt depth
			0pt\vrule height .5pt width 6pt depth 0pt}}\nolimits}
\newfont{\indic}{bbmss12}

\definecolor{light}{gray}{.97}

\title{Variational approximation of functionals defined on 1-dimensional connected sets: the planar case}
\author{M. Bonafini\thanks{Dipartimento di  Matematica, Universit\`a di Trento, Italy, e-mail: mauro.bonafini@unitn.it},
	G. Orlandi\thanks{Dipartimento di Informatica, Universit\`a di Verona, Italy, e-mail: giandomenico.orlandi@univr.it},
	\'E. Oudet\thanks{Laboratoire Jean Kuntzmann, Universit\'e de Grenoble Alpes, France, e-mail: edouard.oudet@imag.fr}}
\date{September 24, 2018}

\begin{document}

\maketitle

\begin{abstract}
In this paper we consider variational problems involving 1-dimensional connected sets in the Euclidean plane, such as the classical Steiner tree problem and the irrigation (Gilbert--Steiner) problem. We relate them to optimal partition problems and provide a variational approximation through Modica--Mortola type energies proving a $\Gamma$-convergence result. We also introduce a suitable convex relaxation and develop the corresponding numerical implementations. The proposed methods are quite general and the results we obtain can be extended to $n$-dimensional Euclidean space or to more general manifold ambients, as shown in the companion paper \cite{BoOrOu}.
\end{abstract}

\section{Introduction}

Connected one dimensional structures play a crucial role in very different areas
like discrete geometry (graphs, networks, spanning and Steiner trees), structural mechanics
(crack formation and propagation), inverse problems (defects identification, contour segmentation), etc. The
modeling of these structures is a key problem both from the theoretical
and the numerical point of view. Most of the difficulties encountered in studying
such one dimensional objects are related to the fact that they are not
canonically associated to standard mathematical quantities. In this article we plan
to bridge the gap between the well-established methods of multi-phase
modeling and the world of one dimensional connected sets or networks. Whereas  we strongly
believe that our approach may lead to new points of view in quite different contexts,
we restrict here our exposition to the  study of two standard  problems
in the Calculus of Variations which are respectively the classical Steiner tree problem and the  Gilbert--Steiner problem (also called the irrigation problem).

The Steiner Tree Problem (STP) \cite{GiPo} can be described as follows: given
$N$ points $P_1, \dots, P_N$ in a metric space $X$ (e.g. $X$ a graph, with
$P_i$ assigned vertices) find a connected graph $F\subset X$ containing
the points $P_i$ and having minimal length. Such an optimal graph $F$ turns out
to be a tree and is thus called a \emph{Steiner Minimal Tree} (SMT).
In case $X=\R^d$, $d\ge 2$ endowed with the Euclidean $\ell^2$ metric, one refers
often to the {\em Euclidean} or {\em geometric} STP, while  for $X=\R^d$ endowed
with the $\ell^1$ (Manhattan) distance or for $X$ contained in a fixed grid
$\mathcal G\subset\R^d$ one refers to the {\em rectilinear} STP.
Here we will adopt the general metric space formulation of \cite{PaSt}: given a
metric space $X$, and given a compact (possibly infinite) set of terminal
points $A\subset X$ , find
$$\label{eq:stp}
\inf \{ {\mathcal H}^1(S) \text{, $S$ connected, } S\supset A \},
\leqno{(\text{STP})}
$$
where ${\mathcal H}^1$ indicates the 1-dimensional Hausdorff measure on $X$.
Existence of solutions for (STP)  relies on Golab's compactness theorem for compact
connected sets, and it holds true also in generalized cases (e.g. $\inf \mathcal H^1(S)$,  $S\cup A$ connected).

The Gilbert--Steiner problem, or $\alpha$-irrigation problem \cite{BeCaMo, Xia} consists in finding a network $S$ along which to flow unit masses located at the sources $P_1, \dots, P_{N-1}$ to the target point $P_N$. Such a network $S$ can be viewed as $S = \cup_{i=1}^{N-1}\gamma_i$, with $\gamma_i$ a path connecting $P_i$ to $P_N$, corresponding to the trajectory of the unit mass located at $P_i$. To favour branching, one is led to consider a cost to be minimized by $S$ which is a sublinear (concave) function of the mass density $\theta(x) = \sum_{i=1}^{N-1} \mathbf{1}_{\gamma_i}(x)$: i.e., for $0\leq\alpha\le 1$, find
$$
\inf \int_S |\theta(x)|^\alpha d{\mathcal H}^1(x).
\leqno{(I_\alpha)}
$$
Notice that $(I_1)$ corresponds to the Monge optimal transport problem, while $(I_0)$ corresponds to (STP). As for (STP) a solution to $(I_\alpha)$ is known to exist and the optimal network $S$ turns out to be a tree \cite{BeCaMo}. 

Problems like (STP) or $(I_\alpha)$ are relevant for the design of optimal transport channels or
networks connecting given endpoints,  for example the optimal design of net
routing in VLSI circuits in the case $d=2,3$.
The Steiner Tree Problem has been widely studied from the theoretical and
numerical point of view in order to efficiently devise constructive solutions,
mainly through combinatoric optimization techniques. Finding a Steiner Minimal
Tree is known to be a NP hard problem (and even NP complete in certain cases),
see for instance \cite{Ar, Ar1} for a comprehensive survey on PTAS algorithms for (STP).

The situation in the Euclidean case for (STP) is theoretically well understood: given $N$
points $P_i\in \R^d$ a SMT connecting them always exists, the solution being in
general not unique (think for instance to symmetric configurations of the
endpoints $P_i$). The SMT is a union of segments connecting the endpoints,
possibly meeting at $120^\circ$ in at most $N-2$ further branch points, called
{\em Steiner points}.

Nonetheless, the quest of computationally tractable approximating schemes for
(STP) and for $(I_\alpha)$ has recently attracted a lot of attention in the Calculus of Variations
community. In particular $(I_\alpha)$ has been studied in the framework of optimal branched transport theory \cite{BeCaMo,BrBuSa}, while (STP) has been interpreted as respectively
a size minimization problem for 1-dimensional connected sets \cite{Mo1,DPHa},
or even a Plateau problem in
a suitable class of vector distributions endowed with some algebraic structure
\cite{Mo1,MaMa}, to be solved by finding suitable {\em calibrations}
\cite{MaOuVe}. Several authors have proposed different approximations of those
problems, whose validity is essentially limited to the planar case, mainly
using a phase field based approach together with some coercive regularization,
see e.g. \cite{BoLeSa,ChMeFe,OuSa,Mi-al}.

Our aim is to propose a variational approximation for (STP) and for the Gilbert--Steiner irrigation problem (in the equivalent formulations of \cite{Xia,MaMa2}) in the Euclidean
case $X=\R^d$, $d\ge 2$.  In this paper we focus on the planar case $d=2$ and
prove a $\Gamma$-convergence result (see Theorem \ref{thm:main} and
Proposition \ref{cor:minconpsi}) by considering integral functionals of
Modica--Mortola type \cite{MoMo}. In the companion paper \cite{BoOrOu} we
rigorously prove that certain integral functionals of Ginzburg-Landau type
(see \cite{ABO2}) yield a variational approximation for (STP) and $(I_\alpha)$ valid in any
dimension $d\ge 3$. This approach is related to the interpretation of (STP) and $(I_\alpha)$ as
a mass minimization problem in a cobordism class of integral currents with multiplicities
in a suitable normed group as studied by Marchese and Massaccesi in \cite{MaMa, MaMa2}
(see also \cite{Mo1} for the planar case). Our method is quite general and
may be easily adapted to a variety of situations (e.g. in manifolds or more
general metric space ambients, with densities or anisotropic norms, etc.).

The plan of the paper is as follows: in Section \ref{secpartition} we reformulate
(STP) and $(I_\alpha)$ as a suitable modification of the optimal partition problem in the planar case. In
Section \ref{secgamma}, we state and prove our main $\Gamma$-convergence results, respectively Proposition \ref{cor:minconpsi} and Theorem \ref{thm:main}.
Inspired by \cite{ChCrPo}, we introduce in Section \ref{secconv} a convex relaxation of the corresponding energies. In
Section \ref{secnum} we present our approximating scheme for (STP) and for the Gilbert-Steiner problem and illustrate
its flexibility in different situations, showing how our convex formulation is
able to recover multiple solutions whereas $\Gamma$-relaxation detects any locally minimizing configuration. Finally, in Section \ref{secgen} we propose some examples and
generalizations that are extensively studied in the companion paper \cite{BoOrOu}.

\section{Steiner problem for Euclidean graphs and optimal partitions\label{secpartition}}

In this section we describe some optimization problems on Euclidean graphs with fixed endpoints set $A$, like (STP) or irrigation-type problems, following the approach of \cite{MaMa,MaMa2}, and we rephrase them as optimal partition-type problems in the planar case $\R^2$.

\subsection{Rank one tensor valued measures and acyclic graphs}\label{sec:intro}
For $M>0$, we consider Radon measures $\Lambda$ on $\R^d$ with values in the space of matrices $\R^{d\times M}$. For each $i = 1,\dots,M$ we define as $\Lambda_i$ the vector measure representing the $i$-th column of $\Lambda$, so that we can write $\Lambda = (\Lambda_1,\dots,\Lambda_{M})$. The total variation measures $|\Lambda_i|$ are defined as usual with respect to the Euclidean structure on $\R^d$, while we set $\mu_\Lambda = \sum_{i=1}^{M} |\Lambda_i|$. Thanks to the Radon--Nikodym theorem we can find a matrix-valued density function $p(x) = (p_1(x), \dots, p_{M}(x))$, with entries $p_{ki} \in L^1(\R^d, \mu_\Lambda)$ for all $k = 1,\dots,d$ and $i = 1,\dots,M$, such that $\Lambda = p(x) \mu_\Lambda$ and $\sum_{i=1}^{M} |p_i(x)| = 1$ for $\mu_\Lambda$-a.e $x \in \R^d$ (where on vectors of $\R^d$ $|\cdot|$ denotes the Euclidean norm). Whenever $p$ is a rank one matrix $\mu_\Lambda$-almost everywhere we say that $\Lambda$ is a rank one tensor valued measure and we write it as $\Lambda = \tau \otimes g \cdot \mu_\Lambda$ for a $\mu_\Lambda$-measurable unit vector field $\tau$ in $\R^d$ and $g \colon \R^d \to \R^{M}$ satisfying $\sum_{i=1}^{M} |g_i| = 1$.

Given $\Lambda \in \mathcal{M}(\R^d, \R^{d\times M})$ and a function $\phi \in C^\infty_c(\R^d; \R^{d\times M})$, with $\phi = (\phi_1, \dots, \phi_{M})$, we have
\[
\langle \Lambda, \phi\rangle = \sum_{i=1}^{M} \langle \Lambda_i, \phi_i \rangle = \sum_{i=1}^{M} \int_{\R^d} \phi_i \,d\Lambda_i,
\]
and fixing a norm $\Psi$ on $\R^{M}$, one may define the $\Psi$-mass measure of $\Lambda$ as
\begin{equation}\label{eq:psitv}
|\Lambda|_\Psi(B):=\sup_{\substack{\omega \in C^\infty_c(B;\R^d) \\ h \in C^{\infty}_c(B;\R^{M})}} \left\{   \langle \Lambda, \omega\otimes h\rangle \, , \quad |\omega(x)|\le 1\, ,\  \Psi^*(h(x))\le 1 \right\}\, ,
\end{equation}
for $B \subset \R^d$ open, where $\Psi^*$ is the dual norm to $\Psi$ w.r.t. the scalar product on $\R^{M}$, i.e. $$\Psi^*(y) = \sup_{x \in \R^{M}} \langle y, x \rangle - \Psi(x).$$ Denote $||\Lambda||_\Psi=|\Lambda|_\Psi(\R^d)$ the $\Psi$-mass norm of $\Lambda$.
In particular one can see that $\mu_\Lambda$ coincides with the measure $|\Lambda|_{\ell^1}$, which from now on will be denoted as $|\Lambda|_1$, and any rank one measure $\Lambda$ may be written as $\Lambda = \tau \otimes g \cdot |\Lambda|_1$ so that $|\Lambda|_{\Psi} = \Psi(g)|\Lambda|_1$. Along the lines of \cite{MaMa} we will rephrase the Steiner and Gilbert--Steiner problem as the optimization of a suitable $\Psi$-mass norm over a given class of rank one tensor valued measures.

Let $A=\{P_1,\dots,P_N\}\subset\R^d$, $d\ge2$, be a given set of $N$ distinct points, with $N>2$. We define the class $\mathcal{G}(A)$ as the set of \emph{acyclic graphs} $L$ connecting the endpoints set $A$ such that $L$ can be described as the union $L=\cup_{i=1}^{N-1} \lambda_i$, where $\lambda_i$ are simple rectifiable curves with finite length having $P_i$ as initial point and $P_N$ as final point, oriented by $\mathcal H^1$-measurable unit vector fields $\tau_i$ satisfying $\tau_i(x)=\tau_j(x)$ for $\mathcal H^1$-a.e. $x\in\, \lambda_i\cap\lambda_j$ (i.e. the orientation of $\lambda_i$ is coherent with that of $\lambda_j$ on their intersection).

For $L \in \mathcal{G}(A)$, if we identify the curves $\lambda_i$ with the vector measures $\Lambda_i=\tau_i\cdot\mathcal{H}^1\res\lambda_i$, all the information concerning this acyclic graph $L$ is encoded in the rank one tensor valued measure $\Lambda=\tau\otimes g\cdot \mathcal H^1\res L$, where the $\mathcal H^1$-measurable vector field $\tau\in\R^d$ carrying the orientation of the graph $L$ satisfies spt$\, \tau=L$,  $|\tau|=1$, $\tau=\tau_i$ $\mathcal H^1$-a.e. on $\lambda_i$, and the $\mathcal H^1$-measurable vector map $g\colon \R^d \to \R^{N-1}$ has components $g_i$ satisfying $g_i\cdot\mathcal H^1\res L=\mathcal H^1\res\lambda_i=|\Lambda_i|$, with $|\Lambda_i|$ the total variation measure of the vector measure $\Lambda_i=\tau\cdot \mathcal H^1\res\lambda_i$. Observe that $g_i\in\{0,1\}$ a.e. for any $1\le i\le N-1$ and moreover that each $\Lambda_i$ verifies the property
\begin{equation}\label{eq:solenoidal}
\div \Lambda_i=\delta_{P_i}-\delta_{P_N}.
\end{equation}

\begin{definition}\label{def:tensor} Given any graph $L \in \mathcal{G}(A)$, we call the above constructed $\Lambda_L \equiv \Lambda=\tau\otimes g\,\cdot\,\mathcal H^1\res L$ the canonical rank one tensor valued measure representation of the acyclic graph $L$.
\end{definition}

To any compact connected set $K\supset A$ with $\mathcal H^1(K)<+\infty$, i.e. to any candidate minimizer for (STP), we may associate in a canonical way an acyclic graph $L \in \mathcal{G}(A)$ connecting $\{ P_1, \dots, P_N \}$ such that $\mathcal H^1(L)\le\mathcal H^1(K)$ (see e.g. Lemma 2.1 in \cite{MaMa}). Given such a graph $L \in \mathcal{G}(A)$ canonically represented by the tensor valued measure $\Lambda$, the measure $\mathcal H^1\res L$ corresponds to the smallest positive measure dominating $\mathcal H^1\res\lambda_i$ for $1\le i\le N-1$.
It is thus given by $\mathcal H^1\res L=\sup_{i}\mathcal H^1\res \lambda_i=\sup_i |\Lambda_i|$, the supremum of the total variation measures $|\Lambda_i|$. We recall that, for any nonnegative $\psi\in C^0_c(\R^d)$, we have
$$
\int_{\R^d}\psi\, d\left( \sup_i |\Lambda_i| \right)=\sup \left\{ \sum_{i=1}^{N-1}\int_{\R^d}\phi_i\, d|\Lambda_i|\, , \ \phi_i\in C^0_c(\R^d), \ \sum_{i=1}^{N-1}\phi_i(x)\le \psi(x)\right\}.
$$

\begin{remark}[graphs as $G$-currents]\label{rem:tensor}{\rm  In \cite{MaMa}, the rank one tensor measure $\Lambda=\tau\otimes g\cdot \mathcal H^1\res L$ identifying a graph in $\R^d$ is defined as a {\em current} with coefficients in the group $\mathbb{Z}^{N-1}\subset \R^{N-1}$. For $\omega\in \mathcal D^1(\R^d)$ a smooth compactly supported differential 1-form and $\vec\phi=(\phi_1,...,\phi_{N-1})\in \mathcal [\mathcal D(\R^d)]^{N-1}$ a smooth test (vector) function, one sets
$$
\begin{aligned}
\langle \Lambda, \omega\otimes\vec\phi\rangle &:=\int_{\R^d}\langle\omega\otimes\vec\phi, \tau\otimes g \rangle\, \,d\mathcal{H}^1\res L=\sum_{i=1}^{N-1}\int_{\R^d}\langle\omega,\tau\rangle\phi_ig_i\,d\mathcal{H}^1\res L \\
&=\sum_{i=1}^{N-1}\int_{\R^d}\langle\omega,\tau\rangle\phi_i\, d|\Lambda_i|\, .
\end{aligned}
$$
Moreover, fixing a norm $\Psi$ on $\R^{N-1}$,  one may define the $\Psi$-mass of the current $\Lambda$ as it is done in \eqref{eq:psitv}. In \cite{MaMa} the authors show that classical integral currents, i.e. $G=\mathbb{Z}$, are not suited to describe (STP) as a mass minimization problem: for example minimizers are not ensured to have connected support.
}
\end{remark}

%
%
\subsection{Irrigation-type functionals}\label{sec:irrigationfunc}

In this section we consider functionals defined on acyclic graphs connecting a fixed set $A=\{P_1,\dots,P_N\}\subset\R^d$, $d\ge 2$, by using their canonical representation as rank one tensor valued measures, in order to identify the graph with an {\em irrigation plan} from the point sources $\{P_1,\dots,P_{N-1}\}$ to the target point $P_N$. We focus here on suitable energies in order to describe the irrigation problem and the Steiner tree problem in a common framework as  in \cite{MaMa,MaMa2}. We observe moreover that the irrigation problem with one point source $(I_\alpha)$ introduced by Xia \cite{Xia}, in the equivalent formulation of \cite{MaMa2}, approximates the Steiner tree problem as $\alpha\to 0$ in the sense of $\Gamma$-convergence (see Proposition \ref{prop:alphatozero}).

Consider on $\R^{N-1}$ the norms $\Psi_\alpha=|\cdot|_{\ell^{1/\alpha}}$ (for $0<\alpha\le 1$) and $\Psi_0=|\cdot|_{\ell^\infty}$. Let $\Lambda=\tau\otimes g \cdot \mathcal H^1\res L$ be the canonical representation of an acyclic graph $L \in \mathcal{G}(A)$, so that we have $|\tau|=1$, $g_i\in\{0,1\}$ for $1\le i\le N-1$ and hence $|g|_\infty = 1$ $\mathcal H^1$-a.e. on $L$. Let us define for such $\Lambda$ and any $\alpha \in [0,1]$ the functional
\[
\mathcal F^\alpha(\Lambda) := ||\Lambda||_{\Psi_\alpha} = |\Lambda|_{\Psi_\alpha}(\R^d).
\]
Observe that, by \eqref{eq:psitv},
\[
\mathcal F^0(\Lambda) = \int_{\R^d} |\tau||g|_\infty \,d\mathcal H^1\res L = \mathcal H^1(L)
\]
and
\begin{equation}\label{eq:Falpha}
\mathcal F^\alpha(\Lambda)=\int_{\R^d} |\tau||g|_{1/\alpha}\,d\mathcal H^1\res L=\int_L |\theta|^\alpha d\mathcal H^1 \, ,
\end{equation}
where $\theta(x)=\sum_i g_i(x)^{1/\alpha}=\sum_i g_i(x)\in\mathbb{Z}$, and $0\le\theta(x)\le N-1$. We thus recognize that minimizing the functional $\mathcal F^\alpha$ among graphs $L$ connecting $P_1,\dots,P_{N-1}$ to $P_N$ solves the irrigation problem $(I_\alpha)$ with unit mass sources $P_1,\dots,P_{N-1}$ and target $P_N$ (see \cite{MaMa2}), while minimizing $\mathcal F^0$ among graphs $L$ with endpoints set $\{P_1,\dots,P_N\}$ solves (STP) in $\R^d$.

Since both $\mathcal F^\alpha$ and $\mathcal F^0$ are mass-type functionals, minimizers do exist in the class of rank one tensor valued measures. The fact that the minimization problem within the class of canonical tensor valued measures representing acyclic graphs has a solution in that class is a consequence of compactness properties of Lipschitz maps (more generally by compactness theorem for $G$-currents \cite{MaMa}; in $\R^2$ it follows alternatively by the compactness theorem in the $SBV$ class \cite{AFP}). Actually, existence of minimizers in the canonically oriented graph class in $\R^2$ can be deduced  as a byproduct of our convergence result (see Proposition \ref{cor:minconpsi} and Theorem \ref{thm:main}) and in $\R^d$, for $d>2$, by the parallel $\Gamma$-convergence analysis contained in the companion paper \cite{BoOrOu}. 

\begin{remark}\label{rem:mincanonic}{\rm A minimizer of $\mathcal F^0$ (resp. $\mathcal F^\alpha$) among tensor valued measures $\Lambda$ representing admissible graphs corresponds necessarily to the canonical representation of a minimal graph, i.e. $g_i \in \{0,1\}$ $\forall\, 1\le i\le N-1$. Indeed since $g_i \in \mathbb{Z}$, if $g_i \neq 0$, we have $|g_i|\geq 1$, hence $g_i \in \{-1,0,1\}$ for minimizers. Moreover if $g_i=-g_j$ on a connected arc in $\lambda_i\cap\lambda_j$,
with $\lambda_i$ going from $P_i$ to $P_N$ and $\lambda_j$ going from $P_j$ to $P_N$, this implies that $\lambda_i\cup\lambda_j$ contains a cycle and $\Lambda$ cannot be a minimizer. Hence, up to reversing the orientation of the graph, $g_i \in \{0,1\}$ for all $1\leq i \leq N-1$.
}
\end{remark}

We conclude this section by observing in the following proposition that the Steiner tree problem can be seen as the limit of irrigation problems. 

\begin{proposition}\label{prop:alphatozero} The functional $\mathcal F^0$ is the $\Gamma$-limit, as $\alpha\to 0$,  of the functionals $\mathcal F^\alpha$ with respect to the convergence of measures.

\end{proposition}

\begin{proof}
Let $\Lambda=\tau \otimes g \cdot \mathcal{H}^1\res L$ be the canonical representation of an acyclic graph $L \in \mathcal{G}(A)$, so that $|\tau|=1$ and $g_i \in \{0,1\}$ for all $i = 1,\dots,N-1$. The functionals $\mathcal F^\alpha(\Lambda)=\int_{\R^d}|g|_{1/\alpha}d\mathcal{H}^1\res L$ generates a monotonic decreasing sequence as $\alpha\to 0$, because $|g|_{p}\le|g|_{q}$ for any $1\le q<p\le +\infty$, and moreover $\mathcal F^\alpha(\Lambda) \to \mathcal F^0(\Lambda)$ because $|g|_{q}\to|g|_{\infty}$ as $q\to +\infty$. Then, by elementary properties of $\Gamma$-convergence (see for instance Remark 1.40 of \cite{Br}) we have $\mathcal F^\alpha\xrightarrow{\Gamma} \mathcal F^0$ . 
\end{proof}

%
%
\subsection{Acyclic graphs and partitions of $\R^2$}
This section is dedicated to the two-dimensional case. 
The aim is to provide an equivalent formulation of (STP) and $(I_\alpha)$ in term of an optimal partition type problem. The equivalence of (STP) with an optimal partition problem has been already studied in the case $P_1,\dots,P_N$ lie on the boundary of a convex set, see for instance \cite{AmBr, AmBr2} and Remark \ref{rem:partitionconvexe}.

To begin we state a result saying that two acyclic graphs having the same endpoints set give rise to a partition of $\R^2$, in the sense that their oriented difference corresponds to the orthogonal distributional gradient of  a piecewise integer valued function having bounded total variation, which in turn determines the partition (see \cite{AFP}). This is actually an instance of the constancy theorem for currents or the Poincar\'e's lemma for distributions (see \cite{FeBook}).

\begin{lemma}\label{lem:partition} Let $\{P, R\}\subset\R^2$ and let $\lambda$, $\gamma$ be simple rectifiable curves from $P$ to $R$ oriented by $\mathcal H^1$-measurable unit vector fields $\tau'$, $\tau''$. Define as above $\Lambda = \tau' \cdot \mathcal{H}^1\res \lambda$ and $\Gamma = \tau'' \cdot \mathcal{H}^1\res \gamma$.

Then there exists a function $u\in SBV(\R^2;\mathbb{Z})$ such that, denoting $Du$ and $Du^{\bot}$ respectively the measures representing the gradient and the orthogonal gradient of $u$, we have $Du^\bot=\Gamma-\Lambda$.
\end{lemma}

\begin{proof}
Consider simple oriented polygonal curves $\lambda_{k}$ and $\gamma_{k}$ connecting $P$ to $R$ such that the Hausdorff distance to respectively $\lambda$ and $\gamma$ is less than $\frac{1}{k}$ and the length of $\lambda_{k}$ (resp. $\gamma_{k}$) converges to the length of $\lambda$ (resp. $\gamma$). We can also assume without loss of generality that $\lambda_k$ and $\gamma_k$ intersect only transversally in a finite number of points $m_k\geq 2$. Let $\tau'_k$, $\tau''_k$ be the $\mathcal H^1$-measurable unit vector fields orienting $\lambda_k$, $\gamma_k$ and define the measures $\Lambda_k = \tau'_k \cdot \mathcal H^1\res\lambda_k$ and $\Gamma_k = \tau_k'' \cdot \mathcal H^1\res \gamma_k$.

For a given $k \in \mathbb N$ consider the closed polyhedral curve $\sigma_{k}=\lambda_{k}\cup\gamma_{k}$ oriented by $\tau_k = \tau_k'-\tau_k''$ (i.e. we reverse the orientation of $\gamma_k$). For every $x \in \R^2\setminus \sigma_k$ let us consider the index of $x$ with respect to $\sigma_k$ (or winding number) and denote it as
\[
u_k(x) = \text{Ind}_{\sigma_k} (x) = \frac{1}{2\pi\textup{i}} \oint_{\sigma_k} \frac{dz}{z-x}.
\]
By Theorem 10.10 in \cite{RuBook}, the function $u_k$ is integer valued and constant in each connected component of $\R^2\setminus\sigma_k$ and vanishes in the unbounded one. Furthermore, for a.e. $x \in \sigma_k$ we have
\[
\lim_{\eps\to0^+} u_k(x + \eps\tau_k(x)^\bot) - \lim_{\eps\to0^-} u_k(x + \eps\tau_k(x)^\bot) = 1,
\]
i.e. $u_k$ has a jump of $+1$ whenever crossing $\sigma_k$ from ``right'' to ``left'' (cf \cite{RoeBook}, Lemma 3.3.2). This means that
\[
Du_k^\bot = -\tau_k \cdot \mathcal{H}^1\res \sigma_k = \Gamma_k - \Lambda_k.
\]
Thus, $|Du_{k}|(\R^2)=\mathcal H^1(\sigma_{k})$ and $\|u_{k}\|_{L^1(\R^2)}\le C|Du_{k}|(\R^2)$ by Poincar\'e's inequality in $BV$. Hence $u_{k}\in SBV(\R^2;\mathbb{Z})$ is an equibounded sequence in norm, and by Rellich compactness theorem there exists a subsequence still denoted $u_{k}$ converging in $L^1(\R^2)$ to a $u\in SBV(\R^2;\mathbb{Z})$. Taking into account that we have $Du_k^\bot = \Gamma_k - \Lambda_k$, we deduce in particular that $Du^\bot = \Gamma - \Lambda$ as desired.
\end{proof}

\begin{remark}\label{rem:graphgamma}{\rm
Let $A\subset\R^2$ as above. For $i=1,...,N-1$ let $\gamma_i$ be the segment joining $P_i$ to $P_N$, denote $\tau_i=\frac{P_N-P_i}{|P_N-P_i|}$ its orientation, and identify $\gamma_i$ with the vector measure $\Gamma_i=\tau_i\cdot\mathcal{H}^1\res\gamma_i$. Then $G=\cup_{i=1}^{N-1}\gamma_i$ is an acyclic graph connecting the endpoints set $A$ and $\mathcal H^1(G)=(\sup_i |\Gamma_i|)(\R^2)$.
}
\end{remark}


Given the set of terminal points $A=\{P_1, \dots , P_N\} \subset \R^2$ let us fix some $G\in\mathcal{G}(A)$ (for example the one constructed in Remark \ref{rem:graphgamma}). For any acyclic graph $L\in\mathcal{G}(A)$, denoting $\Gamma$ (resp. $\Lambda$) the canonical tensor valued representation of $G$ (resp. $L$), by means of Lemma \ref{lem:partition} we have
\begin{equation} \label{eq:drift}
\H^1(L)=\int_{\R^2} \sup_{i} |\Lambda_i| = \int_{\R^2} \sup_i |Du_i^\bot-\Gamma_i|\,
\end{equation}
for suitable $u_i\in SBV(\R^2;\mathbb{Z})$, $1\le i\le N-1$. Thus, using the family of measures $\Gamma = (\Gamma_1, \dots, \Gamma_{N-1})$ of Remark \ref{rem:graphgamma}, we are led to consider the minimization problem for $U \in SBV(\R^2;\mathbb{Z}^{N-1})$ for the functional
\begin{equation}\label{eq:F0full}
F^0(U) = |DU^\bot - \Gamma|_{\Psi_0}(\R^2) = \int_{\R^2} \sup_i |Du_i^\bot-\Gamma_i|.
\end{equation}

\begin{proposition}\label{prop:minF0}
There exists $U \in SBV(\R^2; \mathbb
Z^{N-1})$ such that
\[
F^0(U) = \inf_{V \in SBV(\R^2; \mathbb
Z^{N-1})} F^0(V).
\]
Moreover $\textup{spt}\, U \subset \Omega = \{ x \in \R^2 \,:\, |x|<10\max_i|P_i|\}$.
\end{proposition}

\begin{proof}
Observe first that for any $U \in SBV(\R^2;\mathbb{Z}^{N-1})$ with $F^0(U) < \infty$, we can find $\tilde U$ s.t. $F^0(\tilde U) \leq F^0(U)$ and $\textup{spt}\,\tilde U \subset \Omega$.
Indeed, consider $r = 8\max_i|P_i|$, $\chi = \mathbf{1}_{B_r(0)}$ and $\tilde U = (\chi u_1, \dots, \chi u_{N-1})$. One has, for $1 \leq i \leq N-1$,
\[
\int_{\R^2\setminus B_r(0)} |D\tilde u_i| = \int_{\partial B_r(0)} |u_i^+|
\]
where $u^+_i$ is the trace on $\partial B_r(0)$ of $u_i$ restricted to $B_r(0)$, and
\[
\begin{aligned}
&\int_{\R^2} |D\tilde u_i^\bot - \Gamma_i| = \int_{B_r(0)} |D u_i^\bot - \Gamma_i| + \int_{\partial B_r(0)} |u_i^+| \\
&\leq \int_{B_r(0)} |Du_i^\bot - \Gamma_i| + \int_{\R^2\setminus B_r(0)} |Du_i|
= \int_{\R^2} |Du_i^\bot - \Gamma_i|
\end{aligned}
\]
for any $i = 1,\dots,N-1$, i.e. $F^0(\tilde U) \leq F^0(U)$.

Consider now a minimizing sequence $U^k \in SBV(\R^2; \mathbb{Z}^{N-1})$ of $F^0$. We may suppose w.l.o.g. $\textup{spt}(U^k) \subset \Omega$, so that, for any $1 \leq i \leq N-1$,
\[
|Du_i^k|(\Omega) \leq |Du_i^k-\Gamma_i|(\Omega) + \mathcal H^1(G) \leq F^0(U^k) + \mathcal H^1(G) \leq 3\mathcal H^1(G)
\]
for $k$ sufficiently large. Hence $U^k$ is uniformly bounded in $BV$ by Poincar\'e inequality on $\Omega$, so that it is compact in $L^1(\Omega; \R^{N-1})$ and, up to a subsequence, $U^k \to U$ a.e., whence $U \in SBV(\Omega; \mathbb{Z}^{N-1})$, $\textup{spt}\,U \subset \Omega$ and $U$ minimizes $F^0$ by lower semicontinuity of the norm.
\end{proof}

We have already seen that to each acyclic graph $L\in\mathcal G(A)$ we can associate a function $U \in SBV(\R^2;\mathbb Z^{N-1})$ such that $\mathcal H^1(L)= F^0(U)$. On the other hand, for minimizers of $F^0$, we have the following
\begin{proposition}\label{rem:minrank1}
Let $U\in SBV(\R^2;\mathbb{Z}^{N-1})$ be a minimizer of $F^0$, then there exists an acyclic graph $L \in \mathcal{G}(A)$ connecting the terminal points $P_1, \dots, P_N$ and such that ${F}^0(U) = \mathcal{H}^1(L)$.
\end{proposition}

\begin{proof}
Let $U=(u_1,\dots,u_{N-1})$ be a minimizer of $F^0$ in $SBV(\R^2;\mathbb{Z}^{N-1})$, and denote $\Lambda_i = \Gamma_i - Du_i^\bot$. Observe that each $Du_i$ has no absolutely continuous part with respect to the Lebesgue measure (indeed $u_i$ is piecewise constant being integer valued) and so $\Lambda_i=\tau_i\cdot \mathcal H^1\res\lambda_i$ for some $1$-rectifiable set $\lambda_i$ and $\mathcal{H}^1$-measurable vector field $\tau_i$. Since we have $\div\Lambda_i = \delta_{P_i}-\delta_{P_N}$, $\lambda_i$ necessarily contains a simple rectifiable curve $\lambda'_i$ connecting $P_i$ to $P_N$  (use for instance the decomposition theorem for rectifiable $1$-currents in cyclic and acyclic part, as it is done in \cite{MaMa2}, or the Smirnov decomposition of solenoidal vector fields \cite{Sm}).

Consider now the canonical rank one tensor measure $\Lambda'$ associated to the acyclic subgraph $L'=\lambda_1'\cup \dots \cup\lambda_{N-1}'$ connecting $P_1,\dots,P_{N-1}$ to $P_N$. Then by Lemma \ref{lem:partition}, there exists $U'=(u'_1,\dots,u'_{N-1})\in SBV(\R^2;\mathbb{Z}^{N-1})$ such that
${Du'_i}^\bot=\Gamma_i-\Lambda_i'$ and in particular $F^0(U') = \mathcal{H}^1(L') \leq \mathcal{H}^1(L) \leq F^0(U)$. We deduce $\mathcal H^1(L') = \mathcal H^1(L)$, hence $L' = L$, $L$ is acyclic and $H^1(L) = F^0(U)$.
\end{proof}


\begin{remark}\label{rem:minstructureFpsi}{\rm
We have shown the relationship between (STP) and the minimization of $ F^0$ over functions in $SBV(\R^2; \mathbb Z^{N-1})$, namely
$$\inf \{F^0(U)\,:\, U \in SBV(\R^2;\mathbb{Z}^{N-1}) \} = \inf \{ \mathcal{F}^0(\Lambda_L) \, : \,L \in \mathcal{G}(\{P_1,\dots,P_N\}) \}.$$
A similar connection can be made between the $\alpha$-irrigation problem $(I_\alpha)$ and minimization over $SBV(\R^2; \mathbb Z^{N-1})$ of
\begin{equation}\label{eq:Fafull}
F^\alpha(U) = |DU^\bot - \Gamma|_{\Psi_\alpha}(\R^2),
\end{equation}
namely we have
$$\inf \{F^\alpha(U)\,:\, U \in SBV(\R^2;\mathbb{Z}^{N-1}) \} = \inf \{ \mathcal{F}^\alpha(\Lambda_L) \, : \,L \in \mathcal{G}(\{P_1,\dots,P_N\}) \},$$
where $\mathcal{F}^\alpha$ is defined in equation \eqref{eq:Falpha}. Indeed, given a norm $\Psi$ on $\R^{N-1}$ and $F^\Psi(U) = |DU^\bot - \Gamma|_{\Psi}(\R^2)$ for $U \in SBV(\R^2;\mathbb{Z}^{N-1})$, the proofs of Propositions \ref{prop:minF0} and \ref{rem:minrank1} carry over to this general context: there exists $U \in SBV(\R^2;\mathbb{Z}^{N-1})$ realizing $\inf F^\Psi$, with $\textup{spt}\, U \subset \Omega$ and $DU^\bot - \Gamma = \Lambda_L$ with $\Lambda_L = \tau \otimes g \cdot \mathcal{H}^1\res L$ the canonical representation of an acyclic graph $L \in \mathcal{G}(\{P_1,\dots,P_N\})$.
}
\end{remark}

\begin{remark}\label{rem:partitionconvexe}{\rm
In the case $P_1,\dots,P_N\in\bd\Omega$ with $\Omega\subset\R^2$ a convex set, we may choose $G=\cup_{i=1}^{N-1}\gamma_i$ with $\gamma_i$ connecting $P_i$ to $P_N$ and spt$\, \gamma_i\subset \partial\Omega$. We deduce by \eqref{eq:drift} that for any acyclic graph $L\in\mathcal{G}(A)$
$$
\H^1(L)=\int_{\Omega} \sup_i |Du_i^\bot|\,
$$
for suitable $u_i\in SBV(\Omega;\mathbb{Z})$ such that (in the trace sense) $u_i=1$ on $\gamma_i\subset\bd\Omega$ and $u_i=0$ elsewhere in $\bd\Omega$, $1\le i\le N-1$. We recover here an alternative formulation of the optimal partition problem in a convex planar set $\Omega$ as studied for instance in \cite{AmBr} and \cite{AmBr2}.
}
\end{remark}

The aim of the next section is then to provide an approximation of minimizers of the functionals $F^\alpha$ (and more generally $F^\Psi$) through minimizers of more regular energies of Modica--Mortola type.

\section{Variational approximation of $F^\alpha$\label{secgamma}}

%
%

In this section we state and prove our main results, namely Proposition \ref{cor:minconpsi} and Theorem \ref{thm:main}, concerning the approximation of minimizers of $F^\alpha$ through minimizers of Modica--Mortola type functionals, in the spirit of $\Gamma$-convergence.

\subsection{Modica--Mortola functionals for functions with prescribed jump}\label{subsec:singlejump}

In this section we consider Modica--Mortola functionals for functions having a prescribed jump part along a fixed segment in $\R^2$ and we prove compactness and lower-bounds for sequences having a uniform energy bound. Let $P, Q \in \R^2$ and let $s$ be the segment connecting $P$ to $Q$. We denote by $\tau_s = \frac{Q-P}{|Q-P|}$ its orientation and define $\Sigma_s = \tau_s \cdot \mathcal H^1\res s$. Up to rescaling, suppose $\max(|P|, |Q|) = 1$ and let $\Omega = B_{10}(0)$ and $\Omega_\delta = \Omega \setminus (B_\delta(P) \cup B_\delta(Q))$ for $0 < \delta \ll |Q-P|$. We consider the Modica--Mortola type functionals
\beq\label{eq:momosingle}
F_\eps(u, \Omega_\delta)=\int_{\Omega_\delta} e_\eps(u)\,dx=\int_{\Omega_\delta}\eps|Du^\bot-\Sigma_s|^2+\frac{1}{\eps}W(u)\,dx,
\eeq
defined for $u\in H_s = \{u \in W^{1,2}(\Omega_\delta\setminus s)\cap SBV(\Omega_\delta) \,:\, u|_{\partial \Omega} = 0\}$, where $W$ is a smooth non negative 1-periodic potential vanishing on $\mathbb{Z}$ (e.g. $W(u)=\sin^2(\pi u)$). Define $H(t) = 2\int_0^t \sqrt{W(\tau)}\,d\tau$ and $c_0=H(1)$.

\begin{remark}\label{rem:finiteenergy}{\rm
Notice that any function $u \in H_s$ with $F_\eps(u,\Omega_\delta) < \infty$ has necessarily a prescribed jump $u^+-u^-=+1$ across $s\res \Omega_\delta$ in the direction $\nu_s = -\tau_s^\bot$ in order to erase the contribution of the measure term $\Sigma_s$ in the energy. We thus have the decomposition
$$Du^\bot = \nabla u^\bot\mathcal L^2 + Ju^\bot= \nabla u^\bot\mathcal L^2 + \Sigma_s\res\Omega_\delta,$$
where $\nabla u \in L^2(\Omega_\delta)$ is the absolutely continuous part of $Du$ with respect to the Lebesgue measure $\mathcal{L}^2$, and $Ju = (u^+-u^-)\nu_s \cdot \mathcal{ H}^1\res s = \nu_s\cdot  \mathcal{H}^1\res s$.
}
\end{remark}

\begin{remark}{\rm Notice that we cannot work directly in $\Omega$ with $F_\eps$ due to summability issues around the points $P$ and $Q$ for the absolutely continuous part of the gradient, indeed there are no functions $u \in W^{1,2}(\Omega\setminus s)$ such that $u^+-u^-=1$ on $s$. To avoid this issue one could consider variants of the functionals $F_\eps(\cdot,\Omega)$ by relying on suitable smoothings $\Sigma_{s,\epsilon}=\Sigma_s*\eta_\eps$ of the measure $\Sigma_s$, with $\eta_\eps$ a symmetric mollifier.
}
\end{remark}

\begin{proposition}[Compactness]\label{prop:singlecomp}
For any sequence $\{u_{\eps}\}_\eps \subset H_s$ such that $F_\eps(u_{\eps},\Omega_\delta)\le C$, there exists $u\in SBV(\Omega_\delta;\mathbb{Z})$ such that (up to a subsequence)  $u_{\eps}\to u$ in $L^1(\Omega_\delta)$.
\end{proposition}

\begin{proof}
By Remark \ref{rem:finiteenergy} we have $Du_\eps^\bot = \nabla u_\eps^\bot\mathcal{L}^2 + \Sigma_s\res \Omega_\delta$, and using the classical Modica--Mortola trick one has
\[
\begin{aligned}
C &\geq \int_{\Omega_\delta} \eps |Du_\eps^\bot-\Sigma_s|^2 + \frac1\eps W(u_\eps) \,dx \\
&= \int_{\Omega_\delta} \eps |\nabla u_\eps^\bot|^2 + \frac1\eps W(u_\eps) \,dx \geq 2\int_{\Omega_\delta} \sqrt{W(u_\eps)} |\nabla u_\eps|\,dx.
\end{aligned}
\]
Recall that $H(t) = 2\int_{0}^{t}\sqrt{W(\tau)}\,d\tau$ and $c_0 = H(1)$. By the chain rule, we have
\[
\begin{aligned}
|D(H\circ u_\eps)|(\Omega_\delta) &= 2\int_{\Omega_\delta}\sqrt{W(u_\eps)} |\nabla u_\eps|\,dx + \int_{s}\left(H(u_\eps^+)-H(u_\eps^-)\right) d\mathcal{H}^1(x) \\
&\leq C + c_0\mathcal{H}^1(s).
\end{aligned}
\]
We also have $(H\circ u_\eps)|_{\partial\Omega} = 0$ since $u_\eps$ vanishes on $\partial\Omega$, so that, by the Poincar\'e inequality, $\{ H \circ u_\eps \}_\eps$ is an equibounded sequence in $BV(\Omega_\delta)$, thus compact in $L^1(\Omega_\delta)$. In particular, there exists $v\in L^1(\Omega_\delta)$ such that, up to a subsequence, $H\circ u_\eps\to v$ in $L^1(\Omega_\delta)$ and pointwise a.e. Since $H$ is a strictly increasing continuous function with $c_0(t-1) \leq H(t) \leq c_0(t+1)$ for any $t \in \R$, then $H^{-1}$ is uniformly continuous and $|H^{-1}(t)|\le c_0^{-1}(|t|+1)$ for all $t\in\R$. Hence, up to a subsequence, the family $\{u_\eps\}_\eps\subset L^1(\Omega_\delta)$ is pointwise convergent a.e. to $u=H^{-1}(v)\in L^1(\Omega_\delta)$. By Egoroff's Theorem, for any $\sigma>0$ there exists a measurable $E_\sigma \subset \Omega_\delta$, with $|E_\sigma|<\sigma$, such that $u_\eps\to u$ uniformly in $\Omega_\delta\setminus E_\sigma$. Then, taking into account that $|t|\leq c_0^{-1}(|H(t)|+1)$ for all $t \in \R$, we have
\[
\begin{aligned}
|| u_\eps-u||_{L^1(\Omega_\delta)} &\le ||u_\eps-u||_{L^1(\Omega_\delta\setminus E_\sigma)}+\int_{E_\sigma}(|u_\eps| + |u|)\,dx \\
&\le |\Omega|\, || u_\eps-u||_{L^\infty(\Omega_\delta\setminus E_\sigma)}+2c_0^{-1}|E_\sigma|+c_0^{-1}\int_{E_\sigma} (|H\circ u_\eps| + |v|)\,dx
\end{aligned}
\]
and for $\eps$, $\sigma$ small enough the right hand side can be made arbitrarily small thanks to the uniform integrability of the sequence $\{H\circ u_\eps \}_\eps$. Hence $u_\eps \to u$ in $L^1(\Omega_\delta)$.
Furthermore, by Fatou's lemma we have
\[
\int_{\Omega_\delta} W(u) \, dx \leq \liminf_{\eps\to 0} \int_{\Omega_\delta} W(u_\eps) \, dx \leq \liminf_{\eps\to 0} \eps F_\eps(u_\eps, \Omega_\delta) = 0,
\]
whence $u(x) \in \mathbb{Z}$ for a.e. $x \in \Omega_\delta$. Finally we have
\[
c_0|Du|(\Omega_\delta) = |D(H\circ u)|(\Omega_\delta) \leq \liminf_{\eps\to 0} |D(H\circ u_\eps)|(\Omega_\delta) \leq C + c_0\mathcal{H}^1(s),
\]
i.e. $u \in SBV(\Omega_\delta; \mathbb{Z})$.
\end{proof}

\begin{proposition}[Lower-bound inequality]\label{prop:singleliminf}
Let $\{u_{\eps}\}_\eps \subset H_s$ and $u\in SBV(\Omega_\delta;\mathbb{Z})$ such that $u_{\eps}\to u$ in $L^1(\Omega_\delta)$. Then
\beq\label{eq:gammaliminf}
\liminf_{\eps\to 0}  F_\eps(u_\eps, \Omega_\delta) \ge c_0 |Du^\bot-\Sigma_s|(\Omega_\delta).
\eeq
\end{proposition}

\begin{proof}

\emph{Step 1}. Let us prove first that for any open ball $B \subset \Omega_\delta$ we have
\beq \label{eq:gammaliminfballs}
\liminf_{\eps\to 0}  F_\eps(u_\eps, B) \ge c_0 |Du^\bot-\Sigma_s|(B).
\eeq
We distinguish two cases, according to whether $B\cap s=\emptyset$ or not. In the first case we have
$$
F_\eps(u_\eps, B) =\int_{B}\eps|Du_\eps^\bot|^2+\frac{1}{\eps}W(u_\eps)\,dx.
$$
Reasoning as in the proof of Proposition \ref{prop:singlecomp},
\[
c_0|Du|(B) = |D(H\circ u)|(B) \leq \liminf_{\eps\to 0} |D(H\circ u_\eps)|(B) \leq \liminf_{\eps\to 0} F_\eps(u_\eps,B),
\]
and \eqref{eq:gammaliminfballs} follows.

In the case $B\cap s\neq\emptyset$ we follow the arguments of \cite{BaOr}, and consider $u_0={\bf 1}_{B^+}$, where $B^+=\{z\in B\setminus s \,:\, (z-z_0)\cdot\nu_s>0\}$, for $z_0\in B\cap s$ and $\nu_s^\bot=\tau_s$, so that $Du_0^\bot=\Sigma_s\res B$. Letting $v_\eps=u_\eps-u_0$ we have $Dv_\eps^\bot = Du_\eps^\bot-\Sigma_s = \nabla u_\eps^\bot \mathcal{L}^2$, with $\nabla u_\eps \in L^2(B)$ and $W(v_\eps)=W(u_\eps)$ on $B$ by $1$-periodicity of the potential $W$. Hence
$$
F_\eps(u_\eps,B)=\int_{B}\eps|Dv_\eps|^2+\frac{1}{\eps}W(v_\eps)\,dx.
$$
Let $v = u - u_0$, we have
\[
c_0|Du^\bot-\Sigma_s|(B) = c_0|Dv|(B) \leq \liminf_{\eps\to 0} \int_{B}\eps|Dv_\eps|^2+\frac{1}{\eps}W(v_\eps)\,dx = \liminf_{\eps\to 0} F_\eps(u_\eps,B)
\]
and \eqref{eq:gammaliminfballs} follows.


\emph{Step 2}. Since $|Du^\bot-\Sigma_s|$ is a Radon measure, one has
\begin{equation} \label{eq:regularity}
|Du^\bot-\Sigma_s|(\Omega_\delta) = \sup\left\{ \sum_j |Du^\bot-\Sigma_s|(B_j) \right\}
\end{equation}
where the supremum is taken among all finite collections $\{B_j\}_j$ of pairwise disjoint open balls such that $\cup_j B_j \subset \Omega_\delta$. Applying \eqref{eq:gammaliminfballs} to each $B_j$ and summing over $j$ we have
\[
c_0 \sum_j |Du^\bot-\Sigma_s|(B_j) \leq \sum_j \liminf_{\eps\to 0} F_\eps(u_\eps,B_j) \leq \liminf_{\eps\to 0} \sum_j F_\eps(u_\eps,B_j) \leq \liminf_{\eps\to 0} F_\eps(u_\eps,\Omega_\delta)
\]
which gives \eqref{eq:gammaliminf} thanks to \eqref{eq:regularity}.
\end{proof}

\begin{remark}{\label{rem:weightdliminf}\rm
The proof of Proposition \ref{prop:singleliminf} can be easily adapted to prove a weighted version of \eqref{eq:gammaliminf}: in the same hypothesis, for any non negative $\phi \in C^{\infty}_c(\R^d)$ we have
\[
\liminf_{\eps\to 0} \int_{\Omega_\delta} \phi e_\eps(u_\eps)\,dx \ge c_0 \int_{\Omega_\delta} \phi\,d|Du^\bot-\Sigma_s|.
\]
}
\end{remark}

\begin{remark}{\rm
Proposition \ref{prop:singleliminf} holds true also in case the measure $\Sigma_s$ are associated to oriented simple polyhedral (or even rectifiable) finite length curves joining $P$ to $Q$.
}
\end{remark}

\subsection{The approximating functionals $F_\eps^\Psi$}\label{subsec:muljump}

We now consider Modica--Mortola approximations for $\Psi$-mass functionals such as $F^\alpha$. Let $A = \{P_1, \dots, P_N\}$ be our set of terminal points and $\Psi \colon \R^{N-1}\to [0,+\infty)$ be a norm on $\R^{N-1}$. For any $i\in\{1,\dots,N-1\}$ let $\Gamma_i=\tau_i\cdot \mathcal H^1\res\gamma_i$ be the measure defined in Remark \ref{rem:graphgamma}. Without loss of generality suppose $\max_i(|P_i|) = 1$ and define $\Omega = B_{10}(0)$ and $\Omega_\delta = \Omega \setminus \cup_i B_\delta(P_i)$ for $0 < \delta \ll \min_{ij} |P_i-P_j|$. Let
\beq\label{eq:H}
H_i = \{u \in  W^{1,2}(\Omega\setminus\gamma_i)\cap SBV(\Omega) \,:\,u|_{\partial\Omega} = 0\}, \quad H = H_1\times\dots\times H_{N-1},
\eeq
and for $u \in H_i$ define
\beq
e_\eps^i(u) = \eps|Du^\bot-\Gamma_i|^2+\frac{1}{\eps}W(u).
\eeq
Denote $\vec e_\eps(U) = (e_\eps^1(u_1), \dots, e_\eps^{N-1}(u_{N-1}))$ and consider the functionals
\begin{equation}\label{eq:fpsieps}
{F}_\eps^\Psi(U,\Omega_\delta) = |\vec e_\eps (U) \,dx|_{\Psi}(\Omega_\delta),
\end{equation}
or equivalently, thanks to \eqref{eq:psitv},
\begin{equation}\label{eq:fpsiepsdef}
{ F}_\eps^\Psi(U,\Omega_\delta) = \sup_{\substack{\phi \in C^{\infty}_c(\Omega_\delta;\R^{N-1})}} \left\{ \sum_{i=1}^{N-1} \int_{\Omega_\delta} \phi_i e_\eps^i(u_i)\,dx, \quad \Psi^*(\phi(x))\le 1 \right\}.
\end{equation}
The previous compactness and lower-bound inequality for functionals with a single prescribed jump extend to $F_\eps^\Psi$ as follows.

\begin{proposition}[Compactness]\label{prop:comp}
Given $\{U_{\eps}\}_\eps \subset H$ such that $F^\Psi_\eps(U_{\eps},\Omega_\delta)\le C$, there exists $U\in SBV(\Omega_\delta;\mathbb{Z}^{N-1})$ such that (up to a subsequence) $U_{\eps}\to U$ in $[L^1(\Omega_\delta)]^{N-1}$.
\end{proposition}

\begin{proof}
For each $i = 1,\dots,N-1$, by definition of $F_\eps^\Psi$ we have
\[
\int_{\Omega_\delta} e^i_\eps(u_{\eps,i})\,dx \leq \Psi^*(e_i) F_\eps^\Psi(U_\eps,\Omega_\delta) \leq C \Psi^*(e_i)
\]
and the result follows applying Proposition \ref{prop:singlecomp} componentwise.
\end{proof}

\begin{proposition}[Lower-bound inequality]\label{prop:liminf}
Let $\{U_{\eps}\}_\eps \subset H$ and $U\in SBV(\Omega_\delta;\mathbb{Z}^{N-1})$ such that $U_{\eps}\to U$ in $[L^1(\Omega_\delta)]^{N-1}$. Then
\beq\label{eq:gammaliminfpsi}
\liminf_{\eps\to 0}  F_\eps^\Psi(U_\eps, \Omega_\delta) \ge c_0 |DU^\bot-\Gamma|_\Psi(\Omega_\delta).
\eeq
\end{proposition}

\begin{proof}
Fix $\phi \in C^\infty_c(\Omega_\delta; \R^{N-1})$ with $\phi_i \geq 0$ for any $i = 1,\dots,N-1$ and $\Psi^*(\phi(x)) \leq 1$ for all $x \in \Omega_\delta$. By Remark \ref{rem:weightdliminf} we have
\[
\begin{aligned}
c_0 \sum_{i=1}^{N-1}\int_{\Omega_\delta} \phi_i\,d|Du_{i}^\bot - \Gamma_i| &\leq \sum_{i=1}^{N-1} \liminf_{\eps\to 0} \int_{\Omega_\delta} \phi_i e_\eps^i(u_{\eps,i}) \, dx \leq \liminf_{\eps\to 0} \sum_{i=1}^{N-1} \int_{\Omega_\delta} \phi_i e_\eps^i(u_{\eps,i}) \, dx \\
& \leq \liminf_{\eps\to 0} F_\eps^\Psi(U_\eps,\Omega_\delta),
\end{aligned}
\]
and taking the supremum over $\phi$ we get \eqref{eq:gammaliminfpsi}.


\end{proof}

We now state and prove a version of an upper-bound inequality for the functionals $F_\eps^\Psi$ which will enable us to deduce the convergence of minimizers of $F_\eps^\Psi$ to minimizers of $F^\Psi(U,\Omega_\delta) = c_0|DU^\bot-\Gamma|_{\Psi}(\Omega_\delta)$, for $U \in SBV(\Omega_\delta; \mathbb{Z}^{N-1})$.


\begin{proposition}[Upper-bound inequality]\label{prop:limsup}
Let $\Lambda=\tau\otimes g \cdot \mathcal H^1\res L$ be a rank one tensor valued measure canonically representing an acyclic graph $L \in \mathcal{G}(A)$,
and let $U =(u_1,\dots,u_{N-1}) \in SBV(\Omega_\delta;\mathbb{Z}^{N-1})$
such that $Du^\bot_i=\Gamma_i-\Lambda_i$ for any $i=1,\dots,N-1$. Then there exists a sequence $\{U_\eps\}_\eps \subset H$ such that $U_\eps\to U$ in $[L^1(\Omega_\delta)]^{N-1}$ and
\beq\label{eq:gammalimsup}
\limsup_{\eps\to 0} F_\eps^\Psi(U_\eps, \Omega_\delta) \leq c_0 |DU^\bot-\Gamma|_\Psi(\Omega_\delta).
\eeq
\end{proposition}


\begin{figure}
\centering
\includegraphics[width=0.4\linewidth]{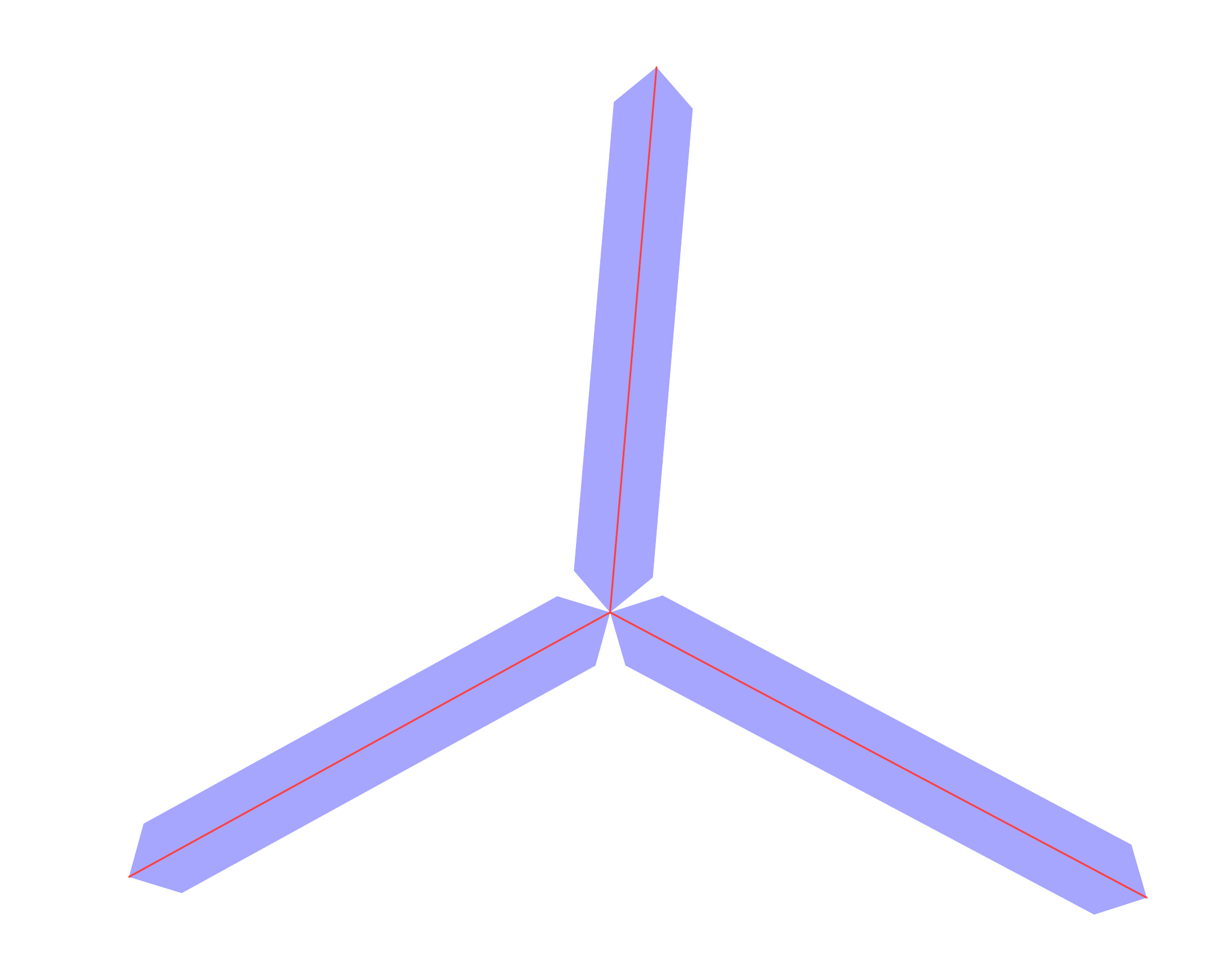}
\qquad
\includegraphics[width=0.45\linewidth]{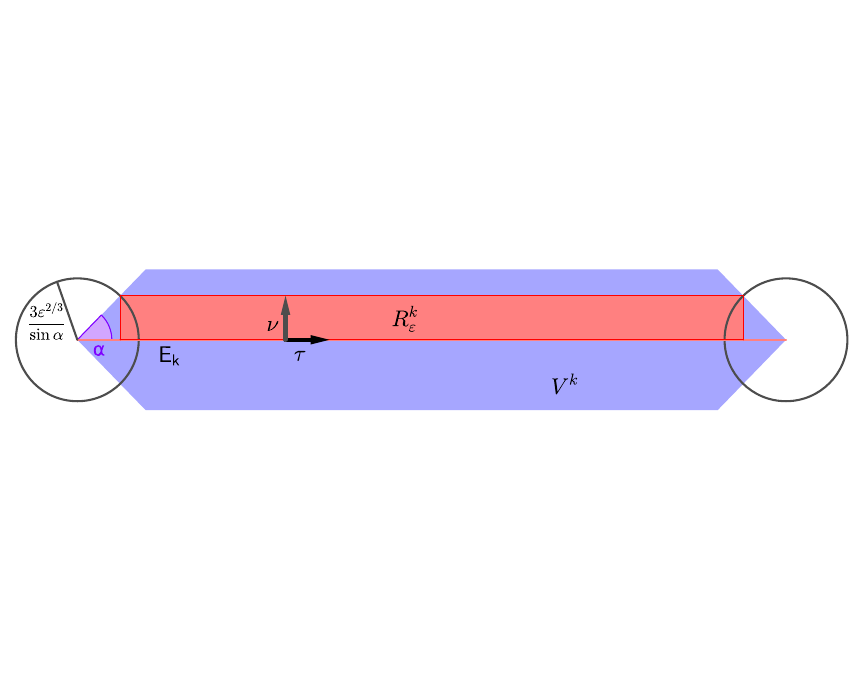}
\caption{Typical shape of the sets $V_k$ (left) and general construction involved in the definition of $R^k_\eps$ (right).}
\label{fig:gammalimsup}
\end{figure}

\begin{proof}
\emph{Step 1.} We consider first the case $\Lambda_i=\tau_i\cdot\mathcal H^1\res\lambda_i$ with $\lambda_i$ a polyhedral curve transverse to $\gamma_i$ for any $1 \leq i < N$. Then the support of the measure $\Lambda$ is an acyclic polyhedral graph (oriented by $\tau$ and with normal $\nu = \tau^\bot$) with edges $E_0, \dots, E_M$ and vertices $\{S_0, \dots, S_\ell\} \nsubseteq (\cup_i\gamma_i) \cap \Omega_\delta$ such that $E_k = [S_{k_1},S_{k_2}]$ for suitable indices $k_1, k_2 \in \{0,\dots,\ell\}$. Denote also $g^k = g|_{E_k} \in \R^{N-1}$ and recall $g^k_i \in \{0,1\}$ for all $1\leq i<N$. By finiteness there exist $\eta>0$ and $\alpha \in (0,\pi/2)$ such that given any edge $E_k$ of that graph the sets
$$
V^k=\{x\in\R^2, {\rm dist}(x, E_k) < \min \{\eta, \, \cos(\alpha)\cdot {\rm dist}(x, S_{k_1}), \, \cos(\alpha)\cdot {\rm dist}(x, S_{k_2})\}\}
$$
are disjoint and their union forms an open neighbourhood of $\cup_i\lambda_i\setminus\{S_0,\dots,S_\ell\}$ (choose for instance $\alpha$ such that $2\alpha$ is smaller than the minimum angle realized by two edges and then pick $\eta$ satisfying $2\eta\tan\alpha < \min_{j} \mathcal{H}^1(E_j)$).


For $0 < \eps \ll \delta$, let $B^m_\eps = \left\{ x \in \R^2 \,:\, |x-S_m| < \frac{3\eps^{2/3}}{\sin\alpha} \right\} $, $B_\eps = \cup_m B_\eps^m$ and define $R^k_\eps \subset V^k$ as
\[
R^k_\eps = \{y + t\nu \,:\, y \in E_k,\, \min\{ {\rm dist}(y, S_{k_1}), {\rm dist}(y, S_{k_2}) \} > 3\eps^{2/3}\cot(\alpha),\, 0 < t \leq 3\eps^{2/3} \}.
\]
Let $\phi_0$ be the optimal profile for the $1$-d Modica--Mortola functional, which solves $\phi_0' = \sqrt{W(\phi_0)}$ on $\R$ and satisfies $\lim_{\tau\to-\infty} \phi_0(\tau) = 0$, $ \lim_{\tau\to\infty} \phi_0(\tau) = 1$ and $\phi_0(0) = 1/2$. Let us define $\tau_\eps = \eps^{-1/3}$, $r_\eps^+ = \phi_0(\tau_\eps)$, $r_\eps^- = \phi_0(-\tau_\eps)$, and
\[
\tilde\phi_\eps(\tau) =
\left\{
\begin{aligned}
& 0 &\quad& \tau < -\tau_\eps-r_\eps^- \\
&\tau + \tau_\eps + r_\eps^- &\quad& -\tau_\eps-r_\eps^- \leq \tau \leq -\tau_\eps \\
& \phi_0(\tau) &\quad& |\tau| \leq \tau_\eps \\
&\tau - \tau_\eps + r_\eps^+ &\quad& \tau_\eps \leq \tau \leq \tau_\eps+1-r_\eps^+ \\
& 1 &\quad& \tau > \tau_\eps+1-r_\eps^+ \\
\end{aligned}
\right.
\]
Observe that $(1-r_\eps^+)$ and $r_\eps^-$ are $o(1)$ as $\eps\to 0$. For $x = y + t\nu \in R^k_\eps$ let us define $\phi_\eps(x) = {\tilde\phi_\eps\left(\frac{t}{\eps} - \tau_\eps - r_\eps^-\right)}$, so that, as $\eps \to 0$,
\[
\begin{aligned}
&\int_{R^k_\eps \cap \Omega_\delta} \eps |D\phi_\eps|^2 + \frac{1}{\eps} W(\phi_\eps)\,dx \leq \mathcal{H}^1(E_k \cap \Omega_\delta) \int_{-\tau_\eps-r_\eps^-}^{2\tau_\eps-r_\eps^-} |D \tilde\phi_\eps(\tau)|^2 + W( \tilde\phi_\eps(\tau)) \,d\tau +o(1) \\
& \leq \mathcal{H}^1(E_k \cap \Omega_\delta) \int_{-\tau_\eps}^{\tau_\eps} 2\phi_0'(\tau)\sqrt{W(\phi_0(\tau))} \,d\tau  + o(1) \leq c_0 \mathcal{H}^1(E_k \cap \Omega_\delta) + o(1).
\end{aligned}
\]
Define, for $x \in \Omega_\delta \setminus B_\eps$,
\[
u_{\eps,i}(x) =
\left\{
\begin{aligned}
& u_i(x) + \phi_\eps(x) - 1 &\quad& \text{if } x \in (R^k_\eps \setminus B_\eps) \cap \Omega_\delta \text{ whenever } E_k \subset \lambda_i\\
& u_i(x) &\quad& \text{elsewhere on } \Omega_\delta \setminus B_\eps \\
\end{aligned}
\right.
\]
and on $B_\eps \cap \Omega_\delta$ define $u_{\eps,i}$ to be a Lipschitz extension of $u_{\eps,i}|_{\partial(B_\eps \cap \Omega_\delta)}$ with the same Lipschitz constant, which is of order $1/\eps$. Remark that $u_{\eps,i}$ has the same prescribed jump as $u_i$ across $\gamma_i$, and thus $F^\Psi_\eps(U_\eps,\Omega_\delta) < \infty$. Moreover $u_{\eps,i} \to u_i$ in $L^1(\Omega_\delta)$.

Observe now that if $E_k$ is contained in $\lambda_i\cap\lambda_j$ then by construction
$$e_{\eps}^i(u_{\eps,i}) = e_{\eps}^j(u_{\eps,j}) = \eps |D\phi_\eps|^2 + \frac{1}{\eps} W(\phi_\eps)$$
on $\tilde R^k_\eps = (R^k_\eps \cap \Omega_\delta) \setminus B_\eps$.
Let $\phi = (\phi_1,\dots,\phi_{N-1})$, with $\phi_i\geq 0$ and $\Psi^*(\phi) \leq 1$, we deduce
\[
\begin{aligned}
&\int_{\Omega_\delta} \sum_i \phi_i e_{\eps}^i(u_{\eps,i}) \,dx \leq \sum_{k=1}^\ell \int_{\tilde R^k_\eps} \sum_i \phi_i e_{\eps}^i(u_{\eps,i}) \,dx + \int_{B_\eps \cap \Omega_\delta} \sum_i \phi_i e_{\eps}^i(u_{\eps,i}) \,dx \\
&\leq \sum_{k=1}^\ell \int_{\tilde R^k_\eps} \sum_i \phi_i g_i^k \left( \eps |D\phi_\eps|^2 + \frac{1}{\eps} W(\phi_\eps) \right) \,dx + \int_{B_\eps \cap \Omega_\delta} \Psi(\vec e_\eps(U_\eps)) \,dx \\
&\leq \sum_{k=1}^\ell \int_{\tilde R^k_\eps} \Psi(g^k) \left( \eps |D\phi_\eps|^2 + \frac{1}{\eps} W(\phi_\eps) \right) \,dx + C\eps^{1/3} \\
&\leq \sum_{k=1}^\ell \Psi(g^k) (c_0\mathcal{H}^1(E_k \cap \Omega_\delta) + o(1)) + C\eps^{1/3} \leq c_0 |DU^\bot - \Gamma|_\Psi(\Omega_\delta) + o(1)\\
\end{aligned}
\]
as $\eps \to 0$. In view of \eqref{eq:fpsiepsdef} we have
\[
F^\Psi_\eps(U_\eps,\Omega_\delta) \leq c_0 |DU^\bot - \Gamma|_\Psi(\Omega_\delta) + o(1),
\]
and conclusion \eqref{eq:gammalimsup} follows.

\medskip

\emph{Step 2.} Let us consider now the case $\Lambda_L \equiv \Lambda = \tau \otimes g \cdot \mathcal{H}^1\res L$, $L = \cup_i \lambda_i$ and the $\lambda_i$ are not necessarily polyhedral. Let $U \in SBV(\Omega_\delta; \mathbb{Z}^{N-1})$ such that $DU^\bot = \Gamma - \Lambda_L$. We rely on Lemma \ref{lemma:polyapprox} below to secure a sequence of acyclic polyhedral graphs $L_n = \cup_i \lambda_{i}^n$, $\lambda_i^n$ transverse to $\gamma_i$, and s.t. the Hausdorff distance $d_H(\lambda_i^n,\lambda_i) < \frac1n$ for all $i = 1,\dots,N-1$, and $|\Lambda_{L_n}|_\Psi(\Omega_\delta) \leq |\Lambda_L|_\Psi(\Omega_\delta) + \frac1n$. Let $U^n \in SBV(\Omega_\delta; \mathbb{Z}^{N-1})$ such that $(DU^n)^\bot = \Gamma - \Lambda_{L_n}$. In particular $U^n \to U$ in $[L^1(\Omega_\delta)]^{N-1}$ and by step 1 we may construct a sequence $U^n_\eps$ s.t. $U^n_\eps \to U^n$ in $[L^1(\Omega_\delta)]^{N-1}$ and
\[
\begin{aligned}
\limsup_{\eps\to 0}  F_\eps^\Psi(U_\eps^n, \Omega_\delta) &\leq c_0 |(DU^n)^\bot-\Gamma|_\Psi(\Omega_\delta) = c_0 |\Lambda_{L_n}|_\Psi(\Omega_\delta) \\ &\leq c_0|\Lambda_L|_\Psi(\Omega_\delta) + \frac{c_0}{n} = c_0|DU^\bot-\Gamma|_\Psi(\Omega_\delta) + \frac{c_0}{n}.
\end{aligned}
\]
We deduce
\[
\limsup_{n\to \infty}  F_{\eps_n}^\Psi(U_{\eps_n}^n, \Omega_\delta) \leq c_0|DU^\bot-\Gamma|_\Psi(\Omega_\delta)
\]
for a subsequence $\eps_n \to 0$ as $n \to +\infty$. Conclusion \eqref{eq:gammalimsup} follows.

\end{proof}

\begin{lemma}\label{lemma:polyapprox}
Let $L \in \mathcal{G}(A)$,  $L = \cup_{i=1}^{N-1} \lambda_i$, be an acyclic graph connecting $P_1,\dots,P_N$. Then for any $\eta > 0$ there exists $L' \in \mathcal{G}(A)$, $L' = \cup_{i=1}^{N-1} \lambda_i'$, with $\lambda'_i$ a simple polyhedral curve of finite length connecting $P_i$ to $P_N$ and transverse to $\gamma_i$, such that the Hausdorff distance $d_H(\lambda_i,\lambda_i') < \eta$ and $|\Lambda_{L'}|_\Psi(\R^2) \leq |\Lambda_L|_\Psi(\R^2) + \eta$, where $\Lambda_{L}$ and $\Lambda_{L'}$ are the canonical tensor valued representations of $L$ and $L'$.
\end{lemma}

\begin{proof}
Since $L \in \mathcal{G}(A)$, we can write $L = \cup_{m=1}^M \zeta_m$, with $\zeta_m$ simple Lipschitz curves such that, for $m_i \neq m_j$, $\zeta_{m_i} \cap \zeta_{m_j}$ is either empty or reduces to one common endpoint. Let $\Lambda_L = \tau\otimes g \cdot \mathcal H^1\res L$ be the rank one tensor valued measure canonically representing $L$, and let $d_m = \Psi(g(x))$ for $x \in \zeta_m$. The $d_m$ are constants because by construction $g$ is constant over each $\zeta_m$. Consider now a polyhedral approximation $\tilde \zeta_m$ of $\zeta_m$ having its same endpoints, with $d_H(\tilde{\zeta}_m, \zeta_m) \leq \eta$, $\mathcal H^1(\tilde\zeta_m) \leq \mathcal H^1(\zeta_m) + \eta/(CM)$ ($C$ to be fixed later) and, for $m_i \neq m_j$, $\tilde\zeta_{m_i} \cap \tilde\zeta_{m_j}$ is either empty or reduces to one common endpoint. Observe that whenever $\zeta_m$ intersects some $\gamma_i$, such a $\tilde{\zeta}_m$ can be constructed in order to intersect $\gamma_i$ transversally in a finite number of points. Define $L' = \cup_{m=1}^M \tilde\zeta_m$ and let $\Lambda_{L'} = \tau' \otimes g'\cdot \mathcal{H}\res L'$ be its canonical tensor valued measure representation. Then, by construction $\Psi(g'(x)) = d_m$ for any $x \in \tilde{\zeta}_m$, hence
\[
|\Lambda_{L'}|_\Psi(\R^2) = \sum_{m=1}^M d_m \mathcal{H}^1(\tilde \zeta_{m}) \leq \sum_{m=1}^M d_m \left(\mathcal H^1(\zeta_m) + \frac{\eta}{CM}\right) \leq |\Lambda_L|_\Psi(\R^2) + \eta,
\]
provided $C = \max\{ \Psi(g) \,:\, g\in\R^{N-1}, g_i \in \{0,1\} \text{ for all } i = 1,\dots,N-1\}$. Remark finally that $d_H(L,L') < \eta$ by construction.
\end{proof}

Thanks to the previous propositions we are now able to prove the following

\begin{proposition}[Convergence of minimizers]\label{cor:minconpsi} Let $\{U_\eps\}_\eps\subset H$ be a sequence of minimizers for $F_\eps^\Psi$ in $H$. Then (up to a subsequence) $U_\eps\to U$ in $[L^1(\Omega_\delta)]^{N-1}$, and $U\in SBV(\Omega_\delta;\mathbb{Z}^{N-1})$ is a minimizer of $F^\Psi(U,\Omega_\delta) = c_0|DU^\bot - \Gamma|_\Psi(\Omega_\delta)$ in $SBV(\Omega_\delta;\mathbb{Z}^{N-1})$.
\end{proposition}

\begin{proof}
Let $V\in SBV(\Omega_\delta;\mathbb{Z}^{N-1})$ such that $DV^\bot = \Gamma-\Lambda$, where $\Lambda$ canonically represents an acyclic graph $L \in \mathcal G(A)$, and let $V_\eps\in H$ such that $\limsup_{\eps\to 0} F_\eps^\Psi(V_\eps,\Omega_\delta )\le F^\Psi (V,\Omega_\delta )$. Since $F_\eps^\Psi(U_\eps,\Omega_\delta )\le F_\eps^\Psi(V_\eps,\Omega_\delta )$, by Proposition \ref{prop:comp} there exists $U \in SBV(\Omega_\delta;\mathbb{Z}^{N-1})$ s.t. $U_\eps \to U$ in $[L^1(\Omega_\delta)]^{N-1}$ and by Proposition \ref{prop:liminf} we have
$$
F^\Psi(U,\Omega_\delta )\le \liminf_{\eps\to 0} F_\eps^\Psi(U_\eps,\Omega_\delta )\le\limsup_{\eps\to 0} F_\eps^\Psi(V_\eps,\Omega_\delta )\le F^\Psi (V,\Omega_\delta )\,.
$$
Given a general $V \in SBV(\Omega_\delta;\mathbb{Z}^{N-1})$ we can proceed like in Remark \ref{rem:minstructureFpsi} and find $V'$ such that $DV'^\bot = \Gamma-\Lambda_{L'}$ with $L'$ acyclic, and $ F^\Psi(V',\Omega_\delta) \leq  F^\Psi(V,\Omega_\delta)$. The conclusion follows.

\end{proof}

Let us focus on the case $\Psi = \Psi_\alpha$, where $\Psi_\alpha(g) = |g|_{1/\alpha}$ for $0<\alpha\leq 1$ and $\Psi_0(g) = |g|_\infty$, and denote $F^0_\eps \equiv F^{\Psi_0}_\eps$ and $F^\alpha_\eps \equiv F^{\Psi_\alpha}_\eps$. For $U=(u_1,\dots,u_{N-1}) \in H$ we have
\begin{equation}\label{eq:Fepsalpha}
{ F}_\eps^0(U,\Omega_\delta)=\int_{\Omega_\delta} \, \sup_{i} e_\eps^i(u_i) \, dx, \quad\quad F_\eps^\alpha(U,\Omega_\delta)=\int_{\Omega_\delta} \, \left(\sum_{i=1}^{N-1} e_\eps^i(u_i)^{1/\alpha}\right)^\alpha\, dx,
\end{equation}
and
\begin{equation}\label{eq:FalphaU}
{F}^0(U,\Omega_\delta):=c_0|DU^\bot-\Gamma|_{\Psi_0}(\Omega_\delta)\, \quad\text{and}\quad {F}^\alpha(U,\Omega_\delta):=c_0|DU^\bot-\Gamma|_{\Psi_\alpha}(\Omega_\delta),
\end{equation}
which are the localized versions of \eqref{eq:F0full} and \eqref{eq:Fafull}.

\begin{theorem}\label{thm:main}
Let $\{P_1,\dots,P_N\} \subset \R^2$ such that $\max_i |P_i|=1$, $0 < \delta \ll \max_{ij}|P_i-P_j|$, $\Omega = B_{10}(0)$ and $\Omega_\delta = \Omega \setminus (\cup_i B_\delta(P_i))$. For $0 \leq \alpha \leq 1$ and $0 < \eps \ll \delta$, denote $F^{\alpha,\delta}_\eps \equiv { F}_\eps^\alpha(\cdot,\Omega_\delta)$ and $F^{\alpha,\delta} \equiv F^\alpha(\cdot,\Omega_\delta)$, with $F^\alpha_\eps(\cdot,\Omega_\delta)$ (resp. $F^\alpha(\cdot,\Omega_\delta)$) defined in \eqref{eq:Fepsalpha} (resp. \eqref{eq:FalphaU}).
\begin{itemize}
\item[(i)] Let $\{ U_\eps^{\alpha,\delta} \}_\eps$ be a sequence of minimizers for $F_\eps^{\alpha,\delta}$ on $H$, with $H$ defined in \eqref{eq:H}. Then, up to subsequences, $U^{\alpha,\delta}_\eps \to U^{\alpha,\delta}$ in $[L^1(\Omega_\delta)]^{N-1}$ as $\eps \to 0$, with $U^{\alpha,\delta} \in SBV(\Omega_\delta;\mathbb{Z}^{N-1})$ a minimizer of $F^{\alpha,\delta}$ on $SBV(\Omega_\delta; \mathbb{Z}^{N-1})$. Furthermore, $F_\eps^{\alpha,\delta}(U_\eps^{\alpha,\delta}) \to F^{\alpha,\delta}(U^{\alpha,\delta})$.

\item[(ii)]Let $\{ U^{\alpha,\delta} \}_\delta$ be a sequence of minimizers for $F^{\alpha,\delta}$ on $SBV(\Omega_\delta; \mathbb{Z}^{N-1})$. Up to subsequences we have $U^{\alpha,\delta} \to U^{\alpha}|_{\Omega_{\eta}}$ in $[L^1(\Omega_{\eta})]^{N-1}$ as $\delta \to 0$ for every fixed $\eta$ sufficiently small, with $U^{\alpha} \in SBV(\Omega;\mathbb{Z}^{N-1})$ a minimizer of $F^{\alpha}$ on $SBV(\Omega; \mathbb{Z}^{N-1})$, and $F^\alpha$ defined in \eqref{eq:F0full}, \eqref{eq:Fafull}. Furthermore, $F^{\alpha,\delta}(U^{\alpha,\delta}) \to F^{\alpha}(U^{\alpha})$.
\end{itemize}
\end{theorem}

\begin{proof}
In view of Proposition \ref{cor:minconpsi} it remains to prove item $(ii)$. The sequence $\{ U^{\alpha,\delta} \}_\delta$ is equibounded in $BV(\Omega_\eta)$ uniformly in $\eta$, hence $U^{\alpha,\delta} \to U$ in $[L^1(\Omega_\eta)]^{N-1}$ for all $\eta>0$ sufficiently small, with $U^\alpha \in SBV(\Omega; \mathbb{Z}^{N-1})$ and $F^{\alpha,\eta}(U^\alpha) \leq \liminf_{\delta \to 0} F^{\alpha,\eta}(U^{\alpha,\delta})$	by lower semicontinuity of $F^{\alpha,\eta}$. On the other hand, let $\bar U^\alpha$ be a minimizer of $F^\alpha$ on $SBV(\Omega;\mathbb{Z}^{N-1})$. We have $F^{\alpha,\eta}(U^{\alpha,\delta}) \leq F^{\alpha,\delta}(U^{\alpha,\delta})$ for any $\delta < \eta$, and by minimality, $F^{\alpha,\delta}(U^{\alpha,\delta}) \leq F^{\alpha,\delta}(\bar U^{\alpha}) \leq F^{\alpha}(\bar U^{\alpha}) \leq F^\alpha(U^\alpha)$. This proves $(ii)$.
\end{proof}

%
%

\section{Convex relaxation\label{secconv}}

In this section we propose convex positively $1$-homogeneous relaxations of the irrigation-type functionals $\mathcal F^\alpha$ for $0\le \alpha <1$ so as to include the Steiner tree problem corresponding to $\alpha=0$ (notice that the case 
$\alpha=1$ corresponds to the well-known Monge-Kantorovich optimal transportation problem with respect to the Monge cost $c(x,y)=|x-y|$).  

More precisely, we consider relaxations of the functional defined by
\[
\mathcal F^\alpha(\Lambda)=\|\Lambda\|_{\Psi_\alpha} = \int_{\R^d} |g|_{1/\alpha} \, d\mathcal{H}^1\res L
\]
if $\Lambda$ is the canonical representation of an acyclic graph $L$ with terminal points $\{P_1,\dots,P_N\}\subset\R^d$, so that in particular, according to Definition \ref{def:tensor}, we can write $\Lambda = \tau \otimes g \cdot \mathcal{H}^1\res L$ with $|\tau|=1$, $g_i \in \{0,1\}$. For any other ${d\times(N-1)}$-matrix valued measure $\Lambda$ on $\R^d$ we set $\mathcal F^\alpha(\Lambda)=+\infty$.

As a preliminary remark observe that, since we are looking for positively $1$-homogeneous extensions, any candidate extension $\mathcal R^\alpha$ satisfies
\[
\mathcal R^\alpha (c\Lambda) = |c|\mathcal F^\alpha(\Lambda)
\]
for any $c \in \R$ and $\Lambda$ of the form $\tau \otimes g \cdot \mathcal{H}^1\res L$ as above. As a consequence we have that $\mathcal R^\alpha(-\Lambda)=\mathcal R^\alpha(\Lambda)$, where $-\Lambda$ represents the same graph $L$ as $\Lambda$ but only with reversed orientation.

\subsection{Extension to rank one tensor measures}\label{sec:rank1}

First of all let us discuss the possible positively $1$-homogeneous convex relaxations of $\mathcal F^\alpha$ on the class of rank one tensor valued Radon measures $\Lambda = \tau\otimes g\cdot|\Lambda|_1$, where $|\tau|=1$, $g\in\R^{N-1}$ (cf. Section \ref{sec:intro}).
For a generic rank one tensor valued measure $\Lambda = \tau\otimes g\cdot|\Lambda|_1$ we can consider extensions of the form
\[
\mathcal R^\alpha(\Lambda) = \int_{\R^d} \Psi^\alpha(g) \,d|\Lambda|_1
\]
for a convex positively 1-homogeneous $\Psi^\alpha$ on $\R^{N-1}$ (i.e. a norm) verifying
\begin{equation}\label{eq:extensionpsi}
\begin{aligned}
& \Psi^\alpha(g) = |g|_{1/\alpha}\, &\qquad&\text{if }\, g_i \in \{0,1\} \text{ for all } i = 1,\dots,N-1, \\
&\Psi^\alpha(g)\geq|g|_{1/\alpha}\, &\qquad&\text{for all }\, g\in\R^{N-1}.
\end{aligned}
\end{equation}
One possible choice is represented by $\Psi^\alpha(g) = |g|_{1/\alpha}$ for all $g \in \R^{N-1}$, while sharper relaxations are given by, for $\alpha>0$,
\begin{equation}\label{eq:amnormalpha}
\Psi^{\alpha}_*(g)=\left(\sum_{1\le i\le N-1}|g_i ^+|^{1/\alpha}\right)^\alpha\,\, + \left(\sum_{1\le i\le N-1} |g_i^-|^{1/\alpha}\right)^\alpha,
\end{equation}
and for $\alpha = 0$ by
\begin{equation}\label{eq:amnorm}
\Psi^{0}_*(g)=\sup_{1\le i\le N-1}g_i ^+\,\, - \inf_{1\le i\le N-1} g_i^-,
\end{equation}
with $g_i^{+}=\max\{g_i,0\}$ and $g_i^-=\min\{g_i,0\}$. In particular $\Psi^{\alpha}_*$ represents the maximal choice within the class of extensions $\Psi^\alpha$ satisfying
\[
\Psi^\alpha(g) = |g|_{1/\alpha} \quad \text{if } g_i \geq 0 \text{ for all } i = 1,\dots,N-1.
\]
Indeed, for $\alpha>0$, $g \in \R^{N-1}$ and $g^\pm = (g^\pm_1,\dots,g^\pm_{N-1})$, we have
\[
\begin{aligned}
\Psi^\alpha(g) &\leq \Psi^\alpha(g^+ + g^-) = 2\Psi^\alpha\left(\frac12 g^+ + \frac12 g^-\right) \leq 2\left( \frac12 \Psi^\alpha(g^+) + \frac12 \Psi^\alpha(g^-) \right) \\
&= \Psi^\alpha(g^+) + \Psi^\alpha(g^-) = |g^+|_{1/\alpha} + |g^-|_{1/\alpha} = \Psi^{\alpha}_*(g).
\end{aligned}
\]
The interest in optimal extensions $\Psi^\alpha$ on rank one tensor valued measures relies in the so called calibration method as a minimality criterion for $\Psi^\alpha$-mass functionals, as it is done in particular in \cite{MaMa} for (STP) using the (optimal) norm $\Psi^0_*$.

According to the convex extensions $\Psi^\alpha$ and $\Psi^0$ considered, when it comes to finding minimizers of respectively $\mathcal R^\alpha$ and $\mathcal R^0$ in suitable classes of weighted graphs with prescribed fluxes at their terminal points, or more generally in the class of rank one tensor valued measures having divergence prescribed by \eqref{eq:solenoidal}, the minimizer is not necessarily the canonical representation of an acyclic graph. Let us consider the following example, where the minimizer contains a cycle.

\medskip

\begin{example}\label{ex:triangle} {\rm  Consider the Steiner tree problem for $\{P_1,P_2,P_3\}\subset\R^2$. We claim that a minimizer of $\mathcal R^0(\Lambda) = \int_{\R^2} |g|_\infty \, d|\Lambda|_1$ within the class of rank one tensor valued Radon measures $\Lambda = \tau\otimes g\cdot|\Lambda|_1$ satisfying \eqref{eq:solenoidal} is supported on the triangle $L=[P_1,P_2]\cup[P_2,P_3]\cup[P_1,P_3]$, hence its support is not acyclic and such a minimizer is not related to any optimal Steiner tree. Denoting $\tau$ the global orientation of $L$ (i.e. from $P_1$ to $P_2$, $P_1$ to $P_3$ and $P_2$ to $P_3$) we actually have as minimizer
\begin{equation}\label{eq:lambdacalibrated}
\Lambda=\tau\otimes\left(\left[\frac12,-\frac12\right]\cdot\mathcal H^1\res[P_1,P_2]+\left[\frac12, \frac12\right]\cdot\mathcal H^1\res[P_3,P_2]+\left[\frac12, \frac12\right]\cdot\mathcal H^1\res[P_3,P_1]\right).
\end{equation}
The proof of the claim follows from Remark \ref{rem:calibration} and Lemma \ref{lem:calibration}.
}
\end{example}

\begin{remark}[Calibrations]\label{rem:calibration}{\rm A way to prove the minimality of $\Lambda=\tau\otimes g \cdot \mathcal H^1\res L$ within the class of rank one tensor valued Radon measures satisfying \eqref{eq:solenoidal} is to exhibit a {\em calibration} for $\Lambda$, i.e. a matrix valued differential form $\omega = (\omega_1, \dots, \omega_{N-1})$, with $\omega_j = \sum_{i=1}^d \omega_{ij}dx_i$ for measurable coefficients $\omega_{ij}$, such that
\begin{itemize}
\item $d\omega_j = 0$ for all $j = 1,\dots,N-1$;
\item $\|\omega\|_*\le 1$, where $\|\cdot\|_*$ is the dual norm to $\|\tau\otimes g\|=|\tau |\cdot |g|_\infty$, defined as
\[
\|\omega\|_* = \sup \{ \tau^t \, \omega \, g \,\,:\,\, |\tau|=1,\; |g|_\infty \leq 1\};
\]
\item $\langle\omega,\Lambda\rangle=\sum_{i,j} \tau_i\omega_{ij}g_j =|g|_\infty$ pointwise, so that 
\[
\int_{\R^2}\langle\omega,\Lambda\rangle = \mathcal R^0(\Lambda).
\]
\end{itemize}
In this way for any competitor $\Sigma=\tau'\otimes g'\cdot|\Sigma|_1$ we have $\langle \omega,\Sigma\rangle \le |g'|_\infty$, and moreover $\Sigma-\Lambda=DU^\perp$, for $U\in BV(\R^2;\mathbb{R}^{N-1})$, hence
$$
\int_{\R^2}\langle\omega,\Lambda-\Sigma\rangle=\int_{\R^2}\langle\omega,DU^\perp \rangle=\int_{\R^2}\langle d\omega,U\rangle=0\, .
$$
It follows
$$
\mathcal R^0(\Sigma)\ge \int_{\R^2} \langle \omega,\Sigma\rangle = \int_{\R^2} \langle \omega,\Lambda\rangle =\mathcal R^0(\Lambda)\, ,
$$
i.e. $\Lambda$ is a minimizer within the given class of competitors.
}
\end{remark}

Let us construct a calibration $\omega=(\omega_1,\omega_2)$ for $\Lambda$ in the general case $P_1\equiv(x_1,0)$, $P_2\equiv(x_2,0)$ and $P_3\equiv(0,x_3)$, with $x_1<0$, $x_1 < x_2$ and $x_3>0$.

\begin{lemma}\label{lem:calibration} Let $P_1$, $P_2$, $P_3$ defined as above and $\Lambda$ as in \eqref{eq:lambdacalibrated}. Consider $\omega=(\omega_1,\omega_2)$ defined as
\[
\begin{aligned}
\omega_1 &= \frac{1}{2a}[(x_1+a) dx + x_3 dy], &\omega_2 = \frac{1}{2a}[(x_1-a) dx + x_3 dy], \quad \text{for }(x,y) \in B_L \\
\omega_1 &= \frac{1}{2b}[(x_2+b) dx + x_3 dy], &\omega_2 = \frac{1}{2b}[(x_2-b) dx + x_3 dy], \quad \text{for }(x,y) \in B_R \\
\end{aligned}
\]
with $B_L$ the left half-plane w.r.t. the line containing the bisector of vertex $P_3$, $B_R$ the corresponding right half-plane and $a = \sqrt{x_1^2+x_3^2}$, $b = \sqrt{x_2^2+x_3^2}$. The matrix valued differential form $\omega$ is a calibration for $\Lambda$.
\end{lemma}

\begin{proof}
For simplicity we consider here the particular case $x_1 = -\frac{1}{2}$, $x_2 = \frac12$ and $x_3 = \frac{\sqrt{3}}{2}$ (the general case is similar). For this choice of $x_1$, $x_2$, $x_3$ we have
\[
\begin{aligned}
\omega_1 &= \frac14 dx+\frac{\sqrt 3}{4} dy, &\omega_2 = -\frac34 dx+\frac{\sqrt 3}{4} dy, \quad \text{for }(x,y) \in \R^2,\, x<0, \\
\omega_1 &= \frac34 dx+\frac{\sqrt 3}{4} dy, &\omega_2 = -\frac14 dx+\frac{\sqrt 3}{4} dy, \quad \text{for }(x,y) \in \R^2,\, x>0. \\
\end{aligned}
\]
The piecewise constant 1-forms $\omega_i$ for $i=1,2$ are globally closed in $\R^2$ (on the line $\{x=0\}$ they have continuous tangential component), $\|\omega\|_*\le 1$ (cf. Remark \ref{rem:calibration}), and taking their scalar product with respectively $(1,0)\otimes(1/2,-1/2)$, $(-1/2,\sqrt{3}/2)\otimes (1/2,1/2)$ for $x<0$ and $(1/2,\sqrt{3}/2)\otimes (1/2,1/2)$ for $x>0$ we obtain in all cases $1/2$, i.e. $|g|_\infty$, so that  
$$
\int_{\R^2} \langle \omega, \Lambda\rangle=\mathcal{R}^0(\Lambda)\,.
$$
Hence $\omega$ is a calibration for $\Lambda$.

\end{proof}

\begin{remark}\label{rem:calibration1} {\rm A calibration always exists for minimizers in the class of rank one tensor valued measures as a consequence of Hahn-Banach theorem (see e.g. \cite{MaMa}), while it may be not the case in general for graphs with integer or real weights. The classical minimal configuration for (STP) with 3 endpoints $P_1$, $P_2$ and $P_3$ admits 
a calibration with respect to the norm $\Psi_{*}^{0}$ in $\R^{N-1}$ (see \cite{MaMa}) and hence it is a minimizer for the relaxed functional $\mathcal R^0(\Lambda)=||\Lambda||_{\Psi^{0}_{*}}$ among all real weighted graphs (and all rank one tensor valued Radon measures satisfying \eqref{eq:solenoidal}). It is an open problem to show whether or not a minimizer of the relaxed functional $\mathcal R^0(\Lambda)=||\Lambda||_{\Psi^{0}_{*}}$ has integer weights.
}
\end{remark}

\subsection{Extension to general matrix valued measures}

Let us turn next to the convex relaxation of $\mathcal F^\alpha$ for generic ${d\times(N-1)}$ matrix valued measures $\Lambda = (\Lambda_1, \dots, \Lambda_{N-1})$, where $\Lambda_i$, for $1\le i\le N-1$, are the vector measures corresponding to the columns of $\Lambda$.
As a first step observe that, due to the positively $1$-homogeneous request on $\mathcal R^\alpha$, whenever $\Lambda = p \cdot \mathcal H^1\res L = \tau \otimes g \cdot \mathcal H^1\res L$, with $|\tau|=cte.$ and $g_i\in\{0,1\}$, we must have
\[
\mathcal R^\alpha(\Lambda) = \int_{\R^d} |\tau||g|_{1/\alpha} \, d\mathcal H^1\res L = \int_{\R^d} \Phi_\alpha(p) \, d\mathcal H^1\res L,
\]
with $\Phi_\alpha(p) = |\tau||g|_{1/\alpha}$ defined only for matrices $p\in K_0$ ($+\infty$ otherwise), where
\[
K_0=\{\tau\otimes g\in\R^{d\times(N-1)},\ g_i\in\{0,1\}, \ \ |\tau|=cte.\}.
\]

Following \cite{ChCrPo}, we look for $\Phi_\alpha^{**}$, the positively 1-homogeneous convex envelope on $\R^{d\times(N-1)}$ of $\Phi_\alpha$.
Setting $q = (q_1, \dots, q_{N-1})$, with $q_i \in \R^d$ its columns, we have that the convex conjugate function $\Phi_\alpha^*(q)=\sup \{ q\cdot p -\Phi_\alpha(p),\ p\in K_0\}$ is given by
\[
\begin{aligned}
\Phi_\alpha^*(q)&=\sup \left\{\, \tau^t\cdot q\cdot g -|\tau|\cdot|g|_{1/\alpha} \, ,\quad |\tau|=cte., \, g=\sum_{i\in J}e_i,\ J\subset\{1,\dots,N-1\} \right\}\\
&= \sup \left\{\, c\left[ \tau^t\cdot \left( \sum_{j\in J} q_j \right) - |J|^\alpha\right] \, ,\quad c\ge 0,\,|\tau|=1, \, J\subset\{1,\dots,N-1\} \right\}.
\end{aligned}
\]
Hence $\Phi_\alpha^*$ is the indicator function of the convex set
\[
K^\alpha=\left\{ q \in \R^{d\times(N-1)},\;\; \left|\sum_{j\in J}q_j \right|\le |J|^\alpha\ \ \forall\, J\subset\{1,\dots,N-1\}\ \right\},
\]
and in particular, for $\alpha=0$, it holds (cf. \cite{ChCrPo})
\[
K^0=\left\{ q \in \R^{d\times(N-1)}, \;\; \left|\sum_{j\in J}q_j \right|\le 1\ \ \forall\, J\subset\{1,\dots,N-1\}\ \right\}.
\]
It follows that $\Phi_\alpha^{**}$ is the support function of $K^\alpha$, i.e., for $p\in \R^{d\times(N-1)}$,
\begin{equation}\label{eq:convenvelopephi}
\Phi_\alpha^{**}(p)=\sup_{q\in K^\alpha} p\cdot q =\sup \left\{ p\cdot q\,, \ \left|\sum_{j\in J} q_j\right|\le |J|^\alpha\, ,\ J\subset\{1,\dots,N-1\}\ \right\}.
\end{equation}
We are then led to consider, for matrix valued test functions $\phi = (\phi_1, \dots,\phi_{N-1})$, the relaxed functional
\[
\mathcal R^\alpha(\Lambda) = \int_{\R^d} \Phi_\alpha^{**} (\Lambda) = \sup \left\{\ \sum_{i=1}^{N-1} \int_{\R^d} \phi_i \,d\Lambda_i,\quad \phi \in C_c^\infty (\R^d; K^\alpha) \right\}.
\]
Observe that for $\Lambda$ a rank one tensor valued measure and $\alpha = 0$ the above expression coincides with the one obtained in the previous section choosing $\Psi^0=\Psi_*^{0}$.

In the planar case $d=2$, consider a $2\times(N-1)$-matrix valued measure $\Lambda = (\Lambda_1,\dots,\Lambda_{N-1})$ such that $\text{div}\, \Lambda_i = \delta_{P_i}-\delta_{P_N}$. Fix a measure $\Gamma$ as for instance in Remark \ref{rem:graphgamma}. We have $\text{div} (\Lambda-\Gamma) = 0$ in $\R^2$ and by Poincar\'e's lemma there exists $U \in BV(\R^2;\R^{N-1})$ such that $\Lambda  = \Gamma - D U^{\perp}$. 
So the relaxed functional reads
\begin{equation}\label{eq:relaxedmu}
\mathcal E^\alpha(U)=\mathcal R^\alpha(\Lambda) \quad\text{for } \Lambda = \Gamma - D U^{\perp}, \, U\in BV(\R^2;\R^{N-1}).
\end{equation}

The relaxed irrigation problem $(I^\alpha)\equiv\min_{BV}\mathcal E^\alpha(U)$ can thus be described in the following equivalent way, according to \eqref{eq:convenvelopephi}: let $q=\phi$ be any matrix valued test function (with columns $q_i=\phi_i$ for $1\le i\le N-1$), then we have
\begin{equation*}\label{eq:relaxed}
(I^\alpha)\equiv\min\limits_{U\in BV(\R^2;\R^{N-1})} \sup \left\{ \int\limits_{\R^2} \sum_{i=1}^{N-1} (Du_i^\bot-\Gamma_i)\cdot\phi_i  \, , \quad \phi\in C^\infty_c(\R^2;K^\alpha) \ \right\}.
\end{equation*}
Notice that with respect to the similar formulation proposed in \cite{ChCrPo}, there is here the presence of an additional ``drift'' term, moreover the constraints set $K^\alpha$ is somewhat different.

\medskip

We compare now the functional $\mathcal E^\alpha(U)$ with the actual convex envelope $(F^\alpha)^{**}(U)$ in the space $BV(\R^2;\R^{N-1})$, where we set $F^\alpha(U) = |DU^\bot-\Gamma|_{\ell^{1/\alpha}}(\R^2)$ if $\Gamma - D U^\perp = \Lambda$ canonically represents an acyclic graph, and $F^\alpha(U) = +\infty$ elsewhere in $BV(\R^2; \R^{N-1})$. In the spirit of  \cite{ChCrPo} (Proposition 3.1), we have

\begin{lemma}\label{lem:chcrpo} We have $\mathcal E^\alpha(U)\le ( F^\alpha)^{**}(U)\le (N-1)^{1-\alpha}\mathcal E^\alpha(U)$ for any $U \in BV(\R^2;\R^{N-1})$ and any $0\le\alpha<1$.\end{lemma}

\begin{proof}
Observe that $\mathcal E^\alpha(U)\le (\mathcal F^\alpha)^{**}(U)$ by convexity of $\mathcal{E}^\alpha (U)$. Moreover, whenever $\Lambda = \Gamma - DU^\perp$ canonically represents a graph connecting $P_1,\dots,P_N$, we have $( F^\alpha)^{**}(U)\le ( F^1)^{**}(U)$ since $ F^\alpha(U)\le  F^1(U)$. For $\alpha>0$, denoting $\Lambda = \Gamma - DU^\perp$, we deduce
$$( F^1)^{**}(U)\leq\sum_{i=1}^{N-1}|\Lambda_i|(\R^d)\le (N-1)^{1-\alpha}\left(\sum_{i=1}^{N-1}|\Lambda_i|^{1/\alpha}\right)^\alpha(\R^d)\le (N-1)^{1-\alpha}\mathcal E^\alpha(U),$$
and analogously we have $( F^1)^{**}(U)\le (N-1)\mathcal E^0(U)$.
\end{proof}

\section{Numerical identification of optimal structures\label{secnum}}

\subsection{Local optimization by $\Gamma$-convergence}
\begin{figure}[tbh]
\centering
\begin{tabular}{cc}
\includegraphics[height=6cm]{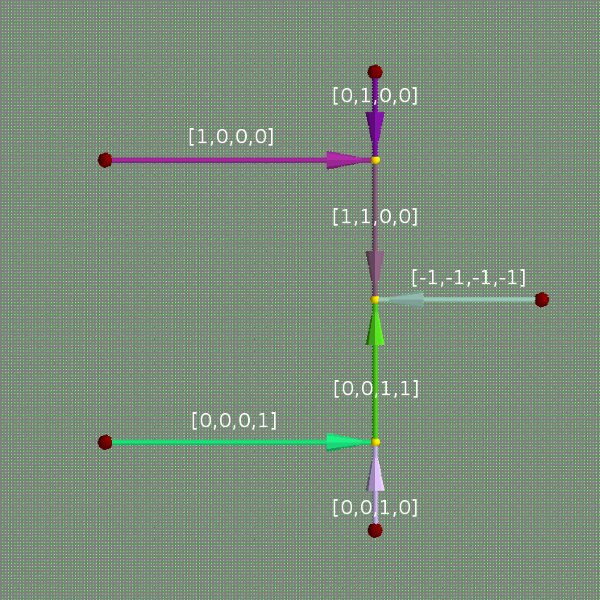} &
\includegraphics[height=6cm]{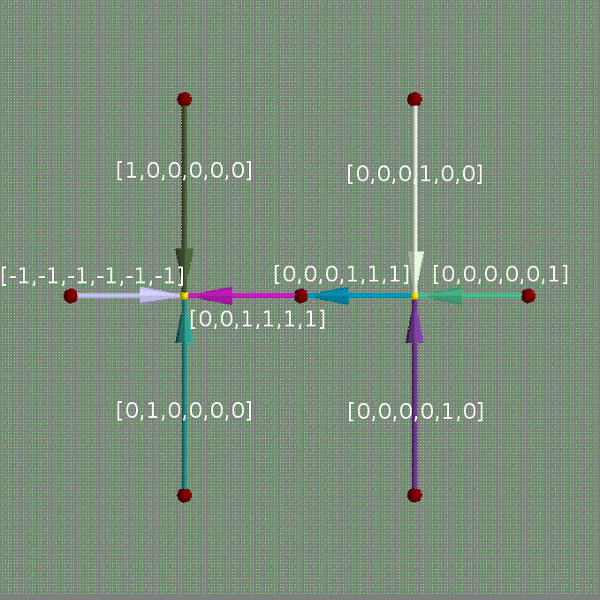} \\
\end{tabular}
\caption{Rectilinear Steiner trees and associated vectorial drifts for five and seven points} \label{fig:drift35}
\end{figure}
In this section, we plan to illustrate the use of Theorem \ref{thm:main}
to identify numerically local minima of the Steiner problem. We base our numerical
approximation on a standard discretization of (\ref{eq:Fepsalpha}). Let $\Omega = (0,1)^2$ and
assume $\{P_1,\dots,P_N\} \subset\Omega$; thus, as a standard consequence, the associated Steiner tree is also contained in
$\Omega$.
Consider a Cartesian grid covering $\Omega$ of step size $h=\frac1S$ where $S>1$ is a fixed integer. Dividing every square cell of the grid into two triangles, we define a triangular mesh $\mathcal{T}$ associated to  $\Omega$ and replace each point $P_i$ with the closest grid point.

Fix now $\Gamma_i$  an oriented vectorial measure absolutely
continuous with respect to $\mathcal H^1$ as in Remark \ref{rem:graphgamma}.
Assume for simplicity  that $\Gamma_i$ is supported  on $\gamma_i$ a union of vertical
and horizontal segments contained in $\Omega$ and covered by the grid
associated to the discrete points $\left\{ (kh,lh),\,0\leq k,l< S \right\}$. Notice
that such a measure can be easily constructed by  considering for instance the
oriented $\ell^1$-spanning tree of the given points.

To mimic the construction in Section \ref{subsec:muljump}, we define the function space
$$H_i^h\equiv P_1(\mathcal{T},\Omega\setminus\gamma_i) \cap BV(\Omega) $$
to be  the set of functions which are globally continuous on $\Omega\setminus\gamma_i$ and
piecewise  linear on every triangle of $\mathcal{T}$. Moreover, we require that
every function of $H_i^h$ has a
jump through $\gamma_i$ of amplitude $-1$ in the orthogonal direction of the orientation
of $\Gamma_i$. Observe that $H_i^h$ is a finite dimensional space of dimension
$S^2$: one element $u^h_i$ can be described by $S^2+n_i$ parameters and $n_i$ linear
constraints describing the jump condition  where $n_i$ is the number of grid points
covered by  $\gamma_i$.

Then, we define
\begin{equation} \label{eq:Fh}
f_h^i(u^h_i) = h|Du^h_i|^2+\frac{1}{h}W(u^h_i),
\end{equation}
if $u \in L^1(\Omega)$ is in $H_i^h$ and extend $f_h^{i}$ by
letting  $f_h^{i}(u) = +\infty$ otherwise. Notice that these discrete energy densities do not contain the drift terms $\Gamma_i$ because the information about the drift has been encoded within the discrete spaces $H_i^h$, leaving us to deal only with the absolutely continuous part of the gradient (see Remark \ref{rem:finiteenergy}). Then, for $U^h = (u_1^h, \dots,u_{N-1}^h) \in H_1^h\times \dots \times H^{h}_{N-1}$ we define
$$
{G}_h^0(U^h)= \int_{\Omega} \sup_{1\le i\le N-1} f^i_h(u_i^h) \quad \text{and} \quad { G}_h^\alpha(U^h)= \int_{\Omega} \, \left(\sum_{i=1}^{N-1} f_h^i(u_i^h)^{1/\alpha}\right)^\alpha.
$$
By a similar strategy we used to prove Theorem \ref{thm:main},
we still also have convergence of minimizers of $G^0_h$ (resp. $G^\alpha_h$) to minimizers of $c_0 F^0$ (resp. $c_0 F^\alpha$) with respect to the strong topology of $L^1(\R^2;\R^{N-1})$. Observe that an exact
evaluation of the integrals involved in \eqref{eq:Fh} is required to obtain this
convergence result (an approximation formula can also be used but then a theoretical proof of convergence would require to study the interaction of the order of approximation with the convergence of minimizers). We point out that this constraint is not critical from
a computational point of view since every function $u_h^i$ of finite energy
has a constant gradient on every triangle of the mesh. On the other hand, the
potential integral can be evaluated formally to obtain an exact estimate of this term
whith respect to the degrees of freedom which describe a function of $H_i^h$.
\begin{figure}[ht]
\centering
\begin{tabular}{ccc}
\includegraphics[height=4cm]{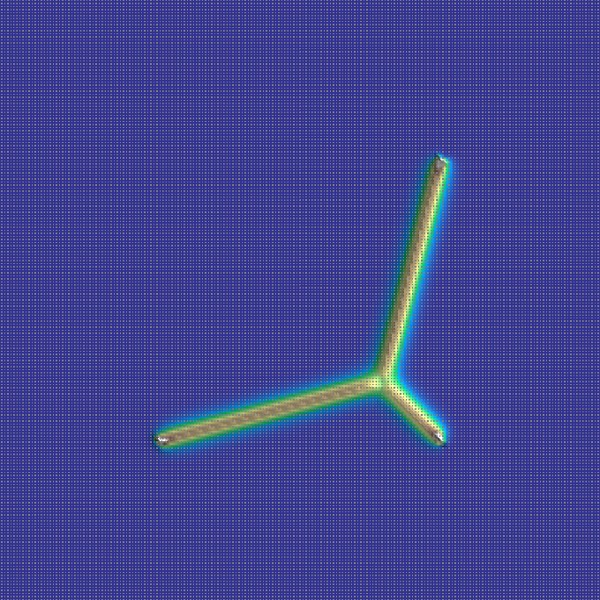} &
\includegraphics[height=4cm]{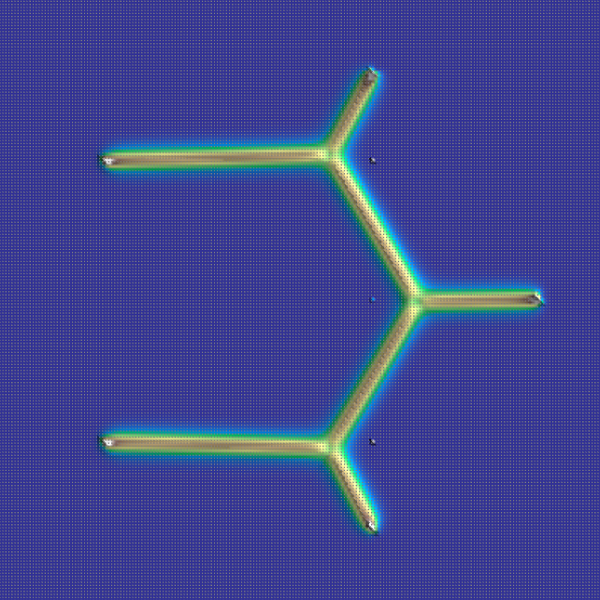} &
\includegraphics[height=4cm]{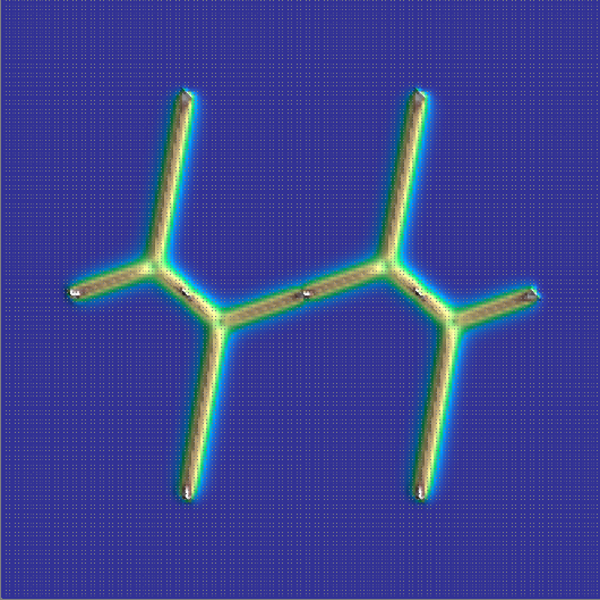}
\end{tabular}
\caption{Local minimizers obtained by the $\Gamma$-convergence approach applied to $3$, $5$ and $7$ points} \label{fig:gamma35}
\end{figure}

Based on these results we performed two different numerical experiments. We first approximated the optimal Steiner trees associated to the vertices of a triangle,
a regular pentagon and a regular hexagon with its center. To obtain the results of figure \ref{fig:gamma35}
we discretized the problem on a grid of size $200 \times 200$. In the case of the triangle
we used the associated spanning tree to define the
measures $(\Gamma_i)_{i = 1,2}$. In the case of the pentagon and of the hexagon
we used the rectilinear Euclidean Steiner trees computed by the Geosteiner's library
(see \cite{warme2001geosteiner} for instance) to initiate the vectorial measures. We refer
to figure \ref{fig:drift35} for an illustration of both singular vector fields.
We solved the resulting finite dimensional problem using an interior point solver. Notice that
in order to deal with the non smooth  cost function  ${ G}_h^0$ we had
to introduce standard gap variables to get a smooth non convex constrained optimization problem. Using
\cite{byrd2006knitro}, we have been able to recover  the locally optimal solutions
depicted in figure  \ref{fig:gamma35} in less than five minutes on a standard computer.
Whereas the results obtained for the triangle and the pentagon describe globally
optimal Steiner trees, the one obtained for the hexagon and its center is only
a local minimizer.

In a second experiment we focus  on simple irrigation problems to illustrate
the versatility of our approach. We  applied exactly the same approach to the pentagon
setting minimizing the functional ${G}_h^\alpha$. We illustrate
our results in figure \ref{fig:irrig5} in which we recover the solutions of
Gilbert-Steiner problems for different values of $\alpha$. Observe that for small
values of $\alpha$, as expected by Proposition \ref{prop:alphatozero}, we recover
an irrigation network close to an optimal Steiner tree.


\begin{figure}[ht]
\centering
\begin{tabular}{ccccc}
\includegraphics[height=2.6cm]{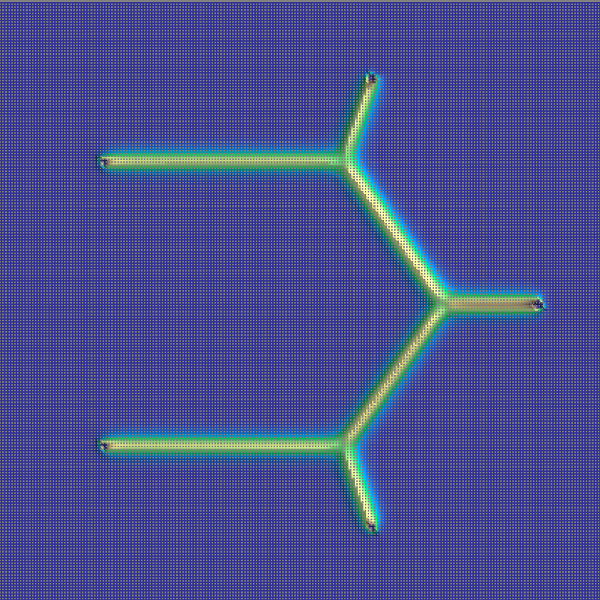} &
\includegraphics[height=2.6cm]{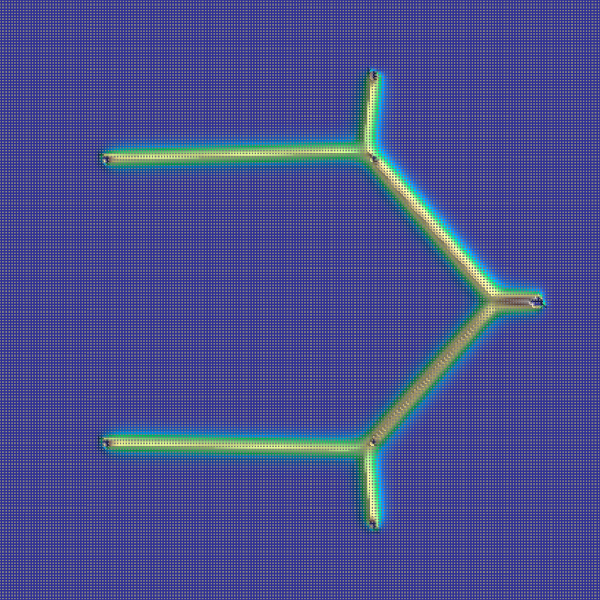} &
\includegraphics[height=2.6cm]{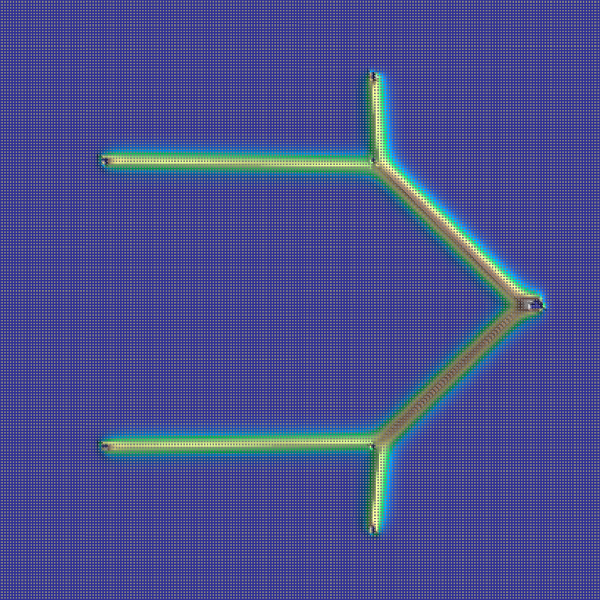} &
\includegraphics[height=2.6cm]{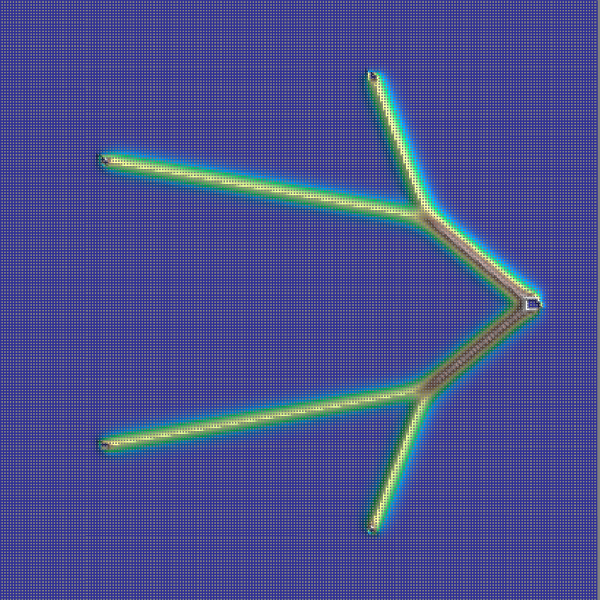} &
\includegraphics[height=2.6cm]{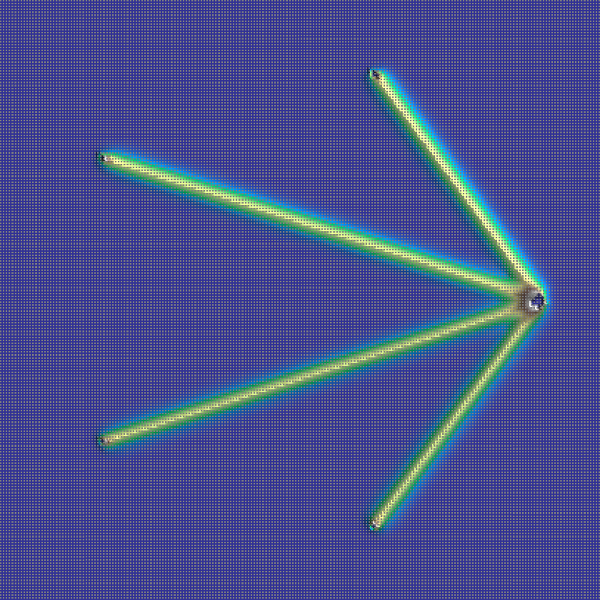}
\end{tabular}
\caption{Gilbert-Steiner solutions associated to parameters $\alpha= 0.2, \, 0.4,\, 0.6,\, 0.8$ and  $1$ (from left to  right)} \label{fig:irrig5}
\end{figure}

\subsection{Convex relaxation and multiple solutions \label{secconv_num}}


The convex relaxation of Steiner problem $(I^0)$  obtained  following \cite{ChCrPo} reads in our discrete setting as:

\begin{equation}
\min_{(u_i^h)_{1\leq i <N}}  \sup_{(\varphi_i^h)_{1\leq i <N} \in K^0
}  \frac{h^2}2 \sum_{t \in \mathcal{T}}  \sum_{i=1}^{N-1}
(\nabla u_i^h)_{t} \cdot (\varphi_i^h)_{t}
\end{equation}
where

\begin{equation} \label{ref:defK}
K^0 = \left\{  (\varphi_i^h)_{1\leq i <N} \in (\R^{2\mathcal{T}})^{N-1}\, | \forall J \subset \{1,\dots,N-1\},
\forall t \in \mathcal{T}, \Big|\sum_{j\in J} (\varphi_j^h)_{t}\Big| \leq 1 \right\}
\end{equation}
and $\forall 1\leq i <N$, $u_i^h \in H_i^h$.
Applying conic duality (see for instance Lecture 2 of \cite{ben2001lectures}), we obtain that the optimal
vector  $(u_i^h)$ solves the following minimization problem

\begin{equation} \label{ref:fconv}
\min_{(u_i^h)_{1\leq i <N} \in L, \, (\psi_J^h)_{ J \subset \{1,\dots,N-1\}} \in (\R^{2\mathcal{T}})^{2^{N-1}} } \frac{h^2}2 \sum_{t \in \mathcal{T}} \sum_{J \subset \{1,\dots,N-1\}} |(\psi_J^h)_{t}|
\end{equation}
where $L$ is the set of discrete vectors $(u_i^h)_{1\leq i <N}$ which satisfy
$\forall i = 1,\dots,N-1, \, \forall t \in \mathcal{T}$:

\begin{equation} \label{ref:defL}
(\nabla u_i^h)_{t}  = \sum_{ J \subset \{1,\dots,N-1\}, \, i \in J} (\psi_J^h)_{t}.
\end{equation}
We solved this convex linearly constrained minimization problem using the conic
solver of the library Mosek  \cite{mosek2010mosek} on a grid of dimension $300 \times300$.
Observe that this convex formulation is also well adapted to the, now standard, large
scale algorithms of proximal type. We studied four different test cases:
the vertices of an equilateral triangle, a square, a pentagon and finally an hexagon
and its center as in previous section.
As illustrated in the left picture of figure \ref{fig:conv3457}, the convex formulation
is able to approximate the optimal structure in the case of the triangle.
Due to the symmetries of the problems, the three last examples do not have unique solutions.
Thus, the result of the optimization is expected to be a convex combination of all solutions whenever the relaxation is sharp, as it can be observed on the second and fourth case of figure \ref{fig:conv3457}. Notice that we do not expect this behaviour to hold for any configuration of points. Indeed the numerical solution in the third picture of figure \ref{fig:conv3457} is not supported on a convex combination of global solutions since the density in the middle point is not $0$.
Whereas the local $\Gamma$-convergence approach of previous
section was only able to produce a local minimum in the case of the hexagon and its center, the convexified formulation gives a relatively precise idea of the set of optimal configurations
(see the last picture of figure \ref{fig:conv3457} where we can recognize within the figure the two global solutions).

\begin{figure}[ht]
\centering
\begin{tabular}{cccc}
\includegraphics[height=3.3cm]{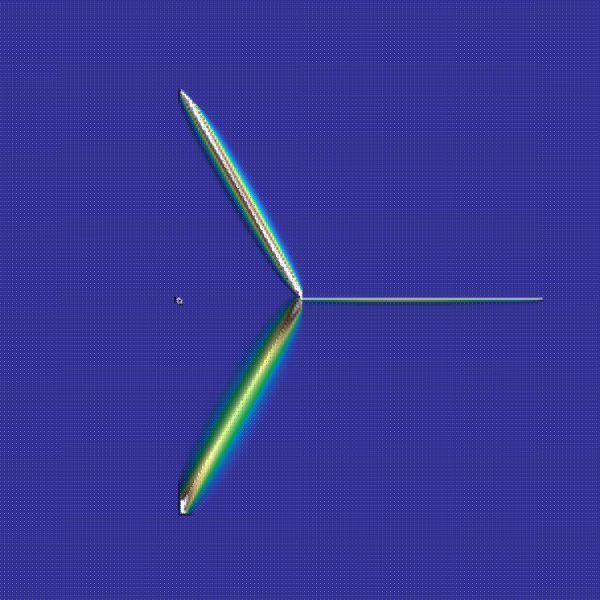} &
\includegraphics[height=3.3cm]{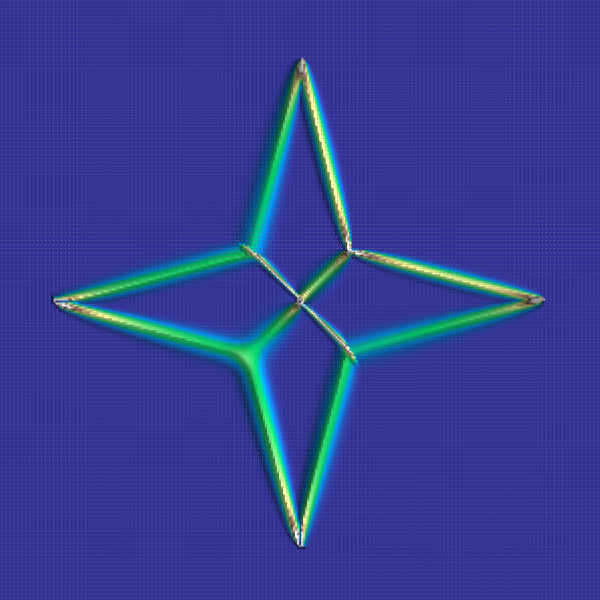} &
\includegraphics[height=3.3cm]{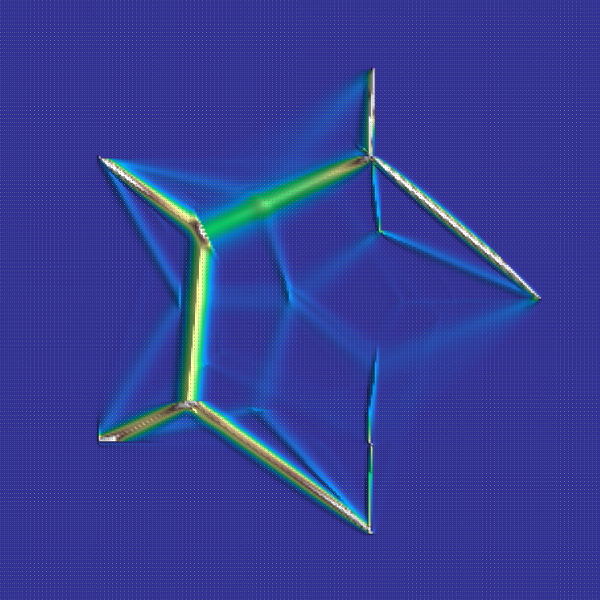} &
\includegraphics[height=3.3cm]{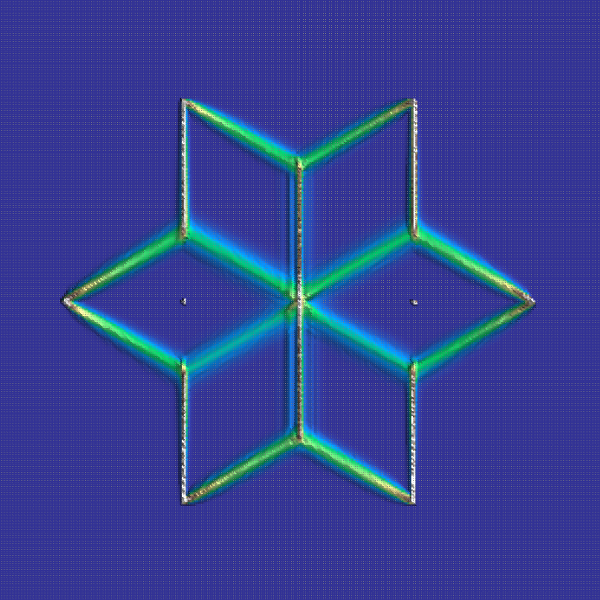}
\end{tabular}
\caption{Results obtained by convex relaxation for $3, 4, \, 5$ and $7$ given points} \label{fig:conv3457}
\end{figure}

\section{Generalizations\label{secgen}}

In this article we have focused on the optimization of one dimensional structures in
the plane in specific, classical cases.
A first possible generalization is to consider the same
problems with respect to more general norms, for instance anisotropic ones:
given $|\cdot|_a$ an anisotropic norm on $\R^d$
and a norm $\Psi^\alpha$ on $\R^{N-1}$ as in Section \ref{sec:rank1}, one could consider, for a matrix valued measure $\Lambda \in \mathcal{M}(\R^d,\R^{d\times N-1})$, the $(\Psi^\alpha,a)$-mass measures
\beq \label{eq:anisotropic}
|\Lambda|_{\Psi^\alpha,a}(B):=\sup_{\substack{\omega \in C^\infty_c(B;\R^d) \\ h \in C^{\infty}_c(B;\R^{N-1})}} \left\{   \langle \Lambda, \omega\otimes h\rangle \, , \quad |\omega(x)|_{a^*} \le 1\, ,\  (\Psi^\alpha)^*(h(x))\le 1 \right\}\, ,
\eeq
for $B \subset \R^d$ open, and the corresponding $(\Psi^\alpha,a)$-mass norm $||\Lambda||_{\Psi^\alpha,a} = |\Lambda|_{\Psi^\alpha,a}(\R^d)$.
Then minimizers of $\mathcal F^\alpha_a = ||\cdot||_{\Psi^\alpha,a}$ over rank one tensor valued measures representing graphs $L \in \mathcal{G}(A)$ will solve the anisotropic irrigation problem (resp. the anisotropic Steiner tree problem in case $\alpha=0$), in particular, if $|\cdot|_a=|\cdot|_1$, $\mathcal F^0_a$ is related to the rectilinear Steiner tree problem in $\R^d$.
For $d=2$, following \cite{Bou,OwSte,AmBel} one may reproduce the $\Gamma$-convergence and convex relaxation approach developed here to numerically study the anisotropic problem \eqref{eq:anisotropic}.
A further step in this direction would consist in considering size or $\alpha$-mass minimization problems in suitable homology and/or oriented cobordism classes for one dimensional chains in manifolds endowed with a Finsler metric.

\medskip
\noindent
Another generalization concerns the convex relaxation and the variational approximation of (STP) and $(I_\alpha)$ in the higher dimensional case $d\ge 3$. This is done in the companion paper \cite{BoOrOu}, where we obtain a $\Gamma$-convergence result by using functionals of Ginzburg-Landau type in the spirit of \cite{ABO2} and \cite{Sandier}. Moreover, as in the present paper, we introduce appropriate ``local" convex envelopes, discuss calibration principles and show some numerical simulations.

\medskip\noindent
In parallel to previous theoretical generalizations,
we are currently developing numerical approaches adapted to these new formulations.
On the one hand, we are studying  a large scale approach to solve
problems analogous to the conic convexified formulation of Section \ref{secconv_num}.
Such an extension is definitely required to approximate realistic problems in dimension
three and higher. On the other hand, we want to focus on refinement techniques which may decrease dramatically
the number of degrees of freedom involved in the optimization process. Observe for instance
that very few parameters are required to describe exactly a drift as the ones given
in Figure \ref{fig:drift35}. Based on such observations, a sequential localized approach
may provide a very precise description of, at least locally, optimal structures.

\section*{Acknowledgements}

The first and second author are partially supported by GNAMPA-INdAM.
The third author gratefully acknowledges the support of the ANR through the
project GEOMETRYA, the project COMEDIC and  the LabEx PERSYVAL-Lab (ANR-11-LABX-0025-01). 
We wish to warmly thank Annalisa Massaccesi and Antonio Marigonda for fruitful discussions.

\bibliographystyle{plain}
\bibliography{references}
\end{document}